\documentclass[11pt]{amsart}
\RequirePackage[l2tabu, orthodox]{nag}
\usepackage[T1]{fontenc}
\usepackage[utf8]{inputenc}
\usepackage{lmodern}
\usepackage[english]{babel}
\usepackage{tikz-cd} 
\usepackage{amsmath}
\usepackage{amsfonts}
\usepackage{amssymb}
\usepackage{enumerate}
\usepackage[alphabetic]{amsrefs}
\usepackage{amsthm}

\usepackage[margin=2.5cm]{geometry}
\newtheorem{theorem}{Theorem}[section]
\newtheorem{lemma}[theorem]{Lemma}

\newtheorem*{claim*}{Claim}

\newtheorem*{theorem*}{Theorem}
\newtheorem{corollary}[theorem]{Corollary}

\theoremstyle{definition}
\newtheorem{definition}[theorem]{Definition}

\theoremstyle{remark}

\newtheorem{remark}[theorem]{Remark}
\newtheorem*{remark*}{Remark}

\numberwithin{equation}{section}
\numberwithin{figure}{section}

\def\NN{{\mathbb N}}

\def\RR{{\mathbb R}}
\newcommand{\T}{\mathbb{S}^1}
\def\ZZ{{\mathbb Z}}

\def\idid{{\mathrm{id}}}

\def\homeo{{\mathrm{Homeo}}}

\def\psl{{\mathrm{PSL}}}
\def\SO{{\mathrm{SO}}}
\def\fix{{\mathrm{Fix}}}
\def\stab{{\mathrm{Stab}}}

\renewcommand{\setminus}{\smallsetminus}
\def\gap{{\mathrm{Gap}}}
\def\core{{\mathrm{Core}}}
\newcommand{\cI}{\mathcal{I}}
\newcommand{\cJ}{\mathcal{J}}

\newcommand{\eqnum}{\refstepcounter{equation}\textup{\tagform@{\theequation}}}

\def\kovacevic{{Kova\v{c}evi\'{c}}}

\usepackage{indentfirst}

\usepackage[colorlinks=true,linkcolor=blue,citecolor=magenta]{hyperref}

\title{Amalgamated free products of circle actions with a bounded number of fixed points}

\author{Jo\~ao Carnevale}
\address{Instituto de Matem\'atica, Universidade Federal do Rio de Janeiro,
	Cidade Universit\'aria -- Ilha do Fund\~ao, Rio de Janeiro, RJ 21945-909, Brazil}
\email{carnevale@im.ufrj.br}

\date{\today}
\subjclass[2020]{Primary 37E10, 57M60; Secondary 37C85, 37B05}
\keywords{group actions on the circle, amalgamated free products, M\"obius-like groups, convergence groups, ping-pong partitions, bounded number of fixed points}

\begin{document}
	
	\begin{abstract}
		Inspired by constructions of Kova\v{c}evi\'{c}, we introduce the amalgamated free product of
		circle actions, obtained by blowing up two actions along prescribed orbits and rearranging
		the inserted intervals. Under natural orbit and index assumptions, we prove that this
		construction is well defined, yields a minimal action on the circle, and is unique up to
		topological conjugacy.
		
		We then study its dynamical properties. Using a proper ping-pong partition arising from the
		construction, we obtain criteria ensuring that the resulting action still has a uniformly
		bounded number of fixed points, and in particular at most \(2n\) fixed points. We also
		give sufficient conditions for the resulting action to remain M\"obius-like and for it not
		to be topologically conjugate to a subgroup of any finite lift \(\psl^{(k)}(2,\RR)\).
	\end{abstract}
	
	\maketitle
	
	\section{Introduction}
	
	Fixed points provide one of the most basic dynamical invariants of a group action on a
	one-manifold. It is therefore natural to ask what can be said about a subgroup of
	\(\homeo_+(\T)\) when every nontrivial element has at most \(N\) fixed points. Before turning
	to the circle case and to the construction developed in this article, let us briefly recall the
	corresponding picture for actions on the line.

	For actions on the line and for very small values of \(N\), there are well-known geometric
	models. H\"older's theorem~\cite{Holder} identifies translations as the basic model for free
	actions; see also \cite[Section~2]{ghys-circle}, while Solodov's theorem~\cite{Solodov} shows that affine actions govern the case of at
	most one fixed point. In a recent extension of Solodov's theorem~\cite{carnevale2022groups},
	the same affine picture was shown to remain valid for actions on the line with at most two fixed
	points: every such action is either elementary or semi-conjugate to an affine action.
	Moreover, under \(C^1\) regularity, one can combine existing rigidity results to promote this semi-conjugacy to a topological	conjugacy in the non-abelian case; see
	\cite[Remark~1.6]{carnevale2022groups}.
	
	On the circle, the corresponding basic examples are given by the rotation group \(\SO(2)\) 
	in the free case and by the M\"obius group
	\(\psl(2,\RR)\) in the case of at most two fixed points. More generally, the finite lifts
	\(\psl^{(k)}(2,\RR)\) provide natural examples in which every nontrivial element has at most \(2k\)
	fixed points. These examples suggest that finite lifts of M\"obius actions should play a central role in
	any attempt to understand circle actions with bounded numbers of fixed points.
	
	However, the circle case is subtler than the line case. A subgroup of \(\homeo_+(\T)\) is called \emph{M\"obius-like} if each of its elements is	individually topologically conjugate to an element of \(\psl(2,\RR)\),and as shown by \kovacevic~\cite{Ko2}, this elementwise condition does not force the whole action to be topologically conjugate to, or even semi-conjugate to, a subgroup of \(\psl(2,\RR)\): there
	exist finitely presented minimal M\"obius-like groups which are not M\"obius groups. More precisely, her construction starts from a M\"obius action and a
	cyclic action generated by a single element, and combines them through a ping-pong argument after
	blowing up a suitable orbit of the M\"obius action. Thus, the obstruction to global M\"obius
	behavior already appears in a very special free-product situation.
	
	This also fits a conjectural picture, attributed to Bonatti, according to which if
	\(G\leq \homeo_+(\T)\) is minimal, every nontrivial element of \(G\) has at most two fixed
	points, and the action of \(G\) is not semi-conjugate to a M\"obius action, then \(G\)
	should split as an amalgamated product over an abelian subgroup; see
	\cite[Conjecture~1.12]{theseJoao}. In a related direction, Kim
	and Triestino~\cite{KimTriestino2025} recently obtained explicit ping-pong partitions
	for hyperbolic-like groups with non-cyclic point stabilizers. The amalgamated free product
	construction developed in this article may be viewed as a systematic framework in that direction.
	
	One of the main purposes of the present article is to formalize and generalize this picture. We
	introduce the \emph{amalgamated free product of circle actions}, obtained by blowing up two circle
	actions along prescribed orbits and rearranging the inserted intervals so that the core of one
	action lies in the complement of the blown-up intervals of the other. Under natural orbit and
	index assumptions, we prove that this construction exists, yields a minimal action on the circle,
	and is unique up to topological conjugacy; see Theorem~\ref{t.action_product}. In this way, the
	amalgamated free product becomes a well-defined dynamical operation, rather than an ad hoc
	surgery.
	
	Once this construction is in place, a first problem is to control the number of fixed points of the
	elements in the resulting group. Using the proper ping-pong partition arising from the construction, 
	we obtain criteria ensuring that
	the amalgamated product still acts with uniformly bounded numbers of fixed points. In particular,
	Theorem~\ref{t.main_examples_2n} shows that, under suitable assumptions on the stabilizers and the
	chosen blown-up orbits, the resulting action has at most \(2n\) fixed points. This provides a general
	mechanism for producing new actions with bounded fixed-point sets from previously given ones.
	
	A second problem is to detect when the resulting action is not topologically conjugate to any finite
	lift \(\psl^{(k)}(2,\RR)\).Here the convergence group theorem of Tukia, Gabai, and Casson--Jungreis
	\cite{Tukia,Gabai,Casson-Jungreis} provides the guiding criterion. Our first obstruction is based on nondiscreteness:
	if one of the factors acts nondiscretely on the circle, then the amalgamated product contains
	sequences whose limit behavior is incompatible with the convergence property, and hence cannot be
	topologically conjugate to any subgroup of \(\psl^{(k)}(2,\RR)\); see
	Theorem~\ref{t.main_examples_non_psl}. This recovers, in a systematic way, the mechanism already
	visible in \kovacevic's examples.
	
	The second obstruction is adapted to the finite-lift setting. We prove a descent principle showing
	that if a minimal amalgamated product is conjugate to a subgroup of \(\psl^{(k)}(2,\RR)\), then the
	same is true for the original factors. We then combine this with a conical-limit-point argument:
	even when one starts from a factor already conjugate to a subgroup of \(\psl^{(k)}(2,\RR)\), the
	presence of a suitable \(k\)-lifted conical limit point in the blown-up orbit forces the resulting
	minimal amalgamated product to fail being topologically conjugate to any finite lift of a M\"obius
	group; see Theorem~\ref{t.conical_examples_non_psl}.
	
	The overall picture is that the amalgamated product provides a structured theory of
	\kovacevic-type examples. Rather than producing isolated counterexamples one at a time, it yields a
	general framework for constructing minimal circle actions with bounded numbers of fixed points and
	with prescribed failures of the convergence property. In this sense, the present work gives further
	evidence for Bonatti's conjectural picture that minimal non-M\"obius actions with at most two fixed
	points should arise from a splitting phenomenon.
	
	The article is organized as follows. In Section~\ref{s.preliminaries} we recall the basic notions on
	semi-conjugacy, blow-ups, and M\"obius-like actions. Section~\ref{s.ping-pong} introduces proper
	ping-pong partitions and proves the fixed-point bounds that will be used later. In
	Section~\ref{s.amalgamated_product} we define the amalgamated free product of circle actions and
	prove its existence and uniqueness up to conjugacy. 
	Finally, the last section is devoted to the dynamical properties of the resulting amalgamated
	free product action. We first study hypotheses ensuring that the action is M\"obius-like, and
	then turn to sufficient conditions guaranteeing that it is not topologically conjugate to any
	subgroup of \(\psl^{(k)}(2,\RR)\), first through nondiscreteness and then through conical limit
	points.

	\section{Preliminaries}\label{s.preliminaries}
	
	Throughout the paper, we write $\T=\RR/\ZZ$ and denote by $\homeo_+(\T)$ the group of
	orientation-preserving homeomorphisms of the circle, endowed with the usual $C^0$ topology.
	
	If $x,y\in\T$ are distinct, we denote by $(x,y)$ the open positively oriented interval from $x$
	to $y$. The notations $[x,y]$, $(x,y]$, and $[x,y)$ are defined similarly. For a subset
	\(A\subset \T\), we denote its interior by \(A^\circ\) and its boundary by \(\partial A\).
	
	Every continuous map $h:\T\to\T$ of degree one admits a lift $\widetilde h:\RR\to\RR$,
	unique up to integer translation, satisfying
	\[
	\widetilde h(x+1)=\widetilde h(x)+1
	\qquad\text{for all }x\in\RR.
	\]
	In particular, the universal cover of $\homeo_+(\T)$ is identified with the group
	$\homeo_{\ZZ}(\RR)$ of homeomorphisms of $\RR$ commuting with integer translations.
	
	A continuous map $h:\T\to\T$ of degree one is called \emph{monotone} if one
	(equivalently, any) of its lifts $\widetilde h:\RR\to\RR$ is non-decreasing.
	We denote by $\gap(h)$ the set of points where $h$ is locally constant, and set
	\[
	\core(h):=\T\smallsetminus \gap(h).
	\]
	
	Let $G$ be a group and let $\varphi,\psi:G\to\homeo_+(\T)$ be two actions on the circle.
	We say that $\varphi$ is \emph{semi-conjugate} to $\psi$ if there exists a continuous monotone
	map $h:\T\to\T$ of degree one such that
	\[
	h\circ \varphi(g)=\psi(g)\circ h
	\qquad\text{for every }g\in G.
	\]
	The map $h$ is called a \emph{semi-conjugacy}. If $h$ is a homeomorphism, then $\varphi$ and
	$\psi$ are \emph{topologically conjugate}.
	
	By abuse of terminology, if $\Gamma_1,\Gamma_2\le \homeo_+(\T)$ are isomorphic subgroups and
	$\rho:\Gamma_1\to\Gamma_2$ is an isomorphism, we will also say that the action of $\Gamma_2$
	is semi-conjugate to the action of $\Gamma_1$ if there exists a continuous monotone map
	$h:\T\to\T$ of degree one such that
	\[
	h\circ \rho(\gamma)=\gamma\circ h
	\qquad\text{for every }\gamma\in \Gamma_1.
	\]
	
	\begin{remark}
		There is a broader notion of semi-conjugacy for circle actions, based on lifts to cyclic central
		extensions, which defines an equivalence relation; see for instance \cite{Kim-Koberda-Mj}.
		However, all semi-conjugacies considered in this article are of the classical form above, namely
		via continuous monotone degree-one maps on the circle.
	\end{remark}
	
	We will also use the standard blow-up terminology associated to semi-conjugacies.
	
	If \(\Gamma_1,\Gamma_2\leq \homeo_+(\T)\) are isomorphic subgroups and the action of
	\(\Gamma_1\) is semi-conjugate to the action of \(\Gamma_2\) by a semi-conjugacy
	\(h:\T\to\T\), then the \(\Gamma_2\)-invariant set \(A:=h(\gap(h))\) records the points of
	\(\Gamma_2\) that have been opened into intervals. In this situation, we say that
	\(\Gamma_1\) is a blow-up of \(\Gamma_2\) on \(A\); see, for instance,
	\cite[Section~2.1]{Kim-Koberda-Denjoy}. More generally, one may prescribe the action of the stabilizer
	of each blown-up orbit on the corresponding interval; this gives a more flexible blow-up
	construction, which exists and is unique up to conjugacy; see \cite[Definition~2.14 and
	Theorem~2.15]{carnevale2022groups}. 
	
	Finally, we recall the terminology related to M\"obius-like actions.
	
	Recall that the standard action of $\psl(2,\RR)$ on the circle identifies $\T$ with the projective
	line $\mathbb{RP}^1$. Every element of $\psl(2,\RR)$ is of exactly one of the following types:
	elliptic, parabolic, or hyperbolic. In particular, an orientation-preserving M\"obius transformation
	has either no fixed point, exactly one fixed point, or exactly two fixed points, and in the latter case, one is attracting and the other is repelling.
	
	A homeomorphism $f\in\homeo_+(\T)$ is \emph{M\"obius-like} if it is topologically conjugate to
	an element of $\psl(2,\RR)$. A subgroup $G\le \homeo_+(\T)$ is \emph{M\"obius-like} if every
	element of $G$ is M\"obius-like.
	
	In particular, a homeomorphism with two fixed points is M\"obius-like if and only if both fixed points are hyperbolic. A homeomorphism with two parabolic fixed points will be called
	\emph{bi-parabolic}. We also stress that a M\"obius-like subgroup need not be conjugate to a
	subgroup of $\psl(2,\RR)$, since the conjugacy may depend on the element; see for instance the
	examples in \cite{Ko2}.
	
	\section{Proper ping-pong partitions}\label{s.ping-pong}
	
	We begin with the following variant of the ping-pong setting.

	\begin{definition}\label{d.ping-pong}
		Let \(X\) be a topological space and let \(F,G\leq \homeo(X)\) be two nontrivial subgroups.
		Set \(S:=F\cap G\), and assume that \(S\) is proper in both \(F\) and \(G\).
		A \emph{proper ping-pong partition} for \(F\) and \(G\) is a pair of disjoint nonempty open
		sets \(U,V\subset X\), each with finitely many connected components, such that:
		\begin{enumerate}[\rm i.]
			\item \((F\smallsetminus S)(U)\subset V\) and \((G\smallsetminus S)(V)\subset U\);
			\item \(S(U)=U\) and \(S(V)=V\).
		\end{enumerate}
	\end{definition}
	
	By the ping-pong lemma for amalgamated free products, if $(U,V)$ is a proper ping-pong partition for $F$ and $G$, then \[\langle F,G\rangle \cong F*_S G.\]
	
	This is a special case of the classical ping-pong argument for amalgamated free products,
	which holds more generally for bijections of arbitrary sets; see
	\cite{Fenchel-Nielsen} and \cite[Section~VII.A]{Maskit}.
	
	The next lemma records a simple dynamical consequence that will be used repeatedly later.
	
	\begin{lemma}\label{l.minimal_ping-pong}
		Assume that \(X\) is compact. With notation as in Definition~\ref{d.ping-pong}, if
		\((U,V)\) is a proper ping-pong partition for \(F\) and \(G\), then
		\(\overline{U\cup V}\) contains a nonempty closed \(\langle F,G\rangle\)-invariant subset.
		In particular, if the action of \(\langle F,G\rangle\) on \(X\) is minimal, then \(\overline{U\cup V}=X\).
	\end{lemma}
	
	\begin{proof}
		Set \(F^*:=F\smallsetminus S\) and \(G^*:=G\smallsetminus S\). Using the ping-pong
		hypotheses, the family of closed sets
		\begin{align*}
			\Lambda_0 &= \overline{U}\cup \overline{V},\\
			\Lambda_1 &= \overline{G^*(V)} \cup \overline{F^*(U)},\\
			\Lambda_2 &= \overline{G^*F^*(U)} \cup \overline{F^*G^*(V)},\\
			\Lambda_3 &= \overline{G^*F^*G^*(V)} \cup \overline{F^*G^*F^*(U)}, \ldots
		\end{align*}
		is nested, hence
		\[
		\Lambda:=\bigcap_{n\in\NN}\Lambda_n
		\]
		is nonempty and closed. Moreover, for every \(f\in F\) and \(g\in G\), one has
		\[
		f(\Lambda_{n+1})\subset \Lambda_n
		\qquad\text{and}\qquad
		g(\Lambda_{n+1})\subset \Lambda_n
		\]
		for all \(n\ge0\). It follows that \(\Lambda\) is invariant under both \(F\) and \(G\), and therefore under \(\langle F,G\rangle\). The last assertion follows
		immediately.
	\end{proof}
	
	From this point on, we restrict to orientation-preserving actions on the circle, that is,
	\(X=\T\) and \(F,G\le \homeo_+(\T)\).
	
	\begin{lemma}\label{l.unique_minimal_ping-pong}
		Let \(F,G\leq \homeo_+(\T)\) admit a proper ping-pong partition \((U,V)\), and assume
		that either \([F:S]>2\) or \([G:S]>2\), where \(S=F\cap G\). Then the action of
		\(\langle F,G\rangle\) on \(\T\) admits a unique minimal set, and this set is infinite.
	\end{lemma}
	
	\begin{proof}
		By Lemma~\ref{l.minimal_ping-pong}, the set \(\overline{U\cup V}\) contains a minimal set, say \(K\). We claim that \(K\) is infinite.
		
		Assume by contradiction that \(K\) is finite. Without loss of generality, suppose that \([G:S]>2\). Set
		\[
		N:=\#(K\cap U),\qquad M:=\#(K\cap V).
		\]
		Choose \(f\in F\smallsetminus S\). Since \(f(U)\subset V\) and \(K\) is \(f\)-invariant, we have
		\[
		f(K\cap U)\subset K\cap V,
		\]
		and therefore \(M\geq N\).
		
		Next choose \(g_1, g_2\in G\smallsetminus S\) belonging to distinct left cosets of \(S\).
		For each \(x\in K\cap V\), the points \(g_1(x)\) and \(g_2(x)\) lie in \(K\cap U\), and
		they are distinct: otherwise \(g_1^{-1}g_2\in G\) would fix a point of \(V\), hence
		\(g_1^{-1}g_2\in S\), a contradiction. Therefore \(N\geq 2M\).
		
		It follows that \(N=M=0\), and therefore \(K\cap (U\cup V)=\emptyset\). By continuity, and
		since the action is orientation-preserving, the same counting argument applies to the closures
		of the connected components of \(U\) and \(V\). Thus
		\[
		K\cap (\overline U\cup \overline V)=\emptyset,
		\]
		contradicting the inclusion \(K\subset \overline{U\cup V}\). Hence \(K\) is infinite.
		
		Finally, an orientation-preserving action on the circle with an infinite minimal set admits a
		unique minimal set; see \cite[Proposition~5.6]{ghys-circle}.
	\end{proof}

	
	\section{Amalgamated products of circle actions}\label{s.amalgamated_product}
	
	In this section we define the amalgamated product of two circle actions and prove that,
	under a natural set of hypotheses, it exists and is unique up to conjugacy.

	\subsection{Setup and definition}\label{d.action_product}
	
	Let \(F,G\leq \homeo_+(\T)\) be countable subgroups, and let
	\(\bar x=(x_1,\ldots,x_n)\) and \(\bar y=(y_1,\ldots,y_n)\) be cyclically ordered
	\(n\)-tuples of points in \(\T\). Assume that, for all distinct \(i,j\in\{1,\ldots,n\}\),
	one has \(x_i\notin F.x_j\) and \(y_i\notin G.y_j\), and that
	\(\stab_F(x_i)=S_F\) and \(\stab_G(y_i)=S_G\), where
	\(\theta:S_F\xrightarrow{\sim}S_G\) is a fixed isomorphism. Let
	\(\sigma:\{1,\ldots,n\}\to\{1,\ldots,n\}\) be a cyclic permutation.
	
	We say that a subgroup \(H\leq \homeo_+(\T)\) is an \emph{amalgamated product of \(F\) and \(G\)
		on \(\bar x\) and \(\bar y\) by the isomorphism \(\theta\) and the permutation \(\sigma\)} if
	\(H\) contains subgroups \(\Gamma_F,\Gamma_G\leq H\) satisfying:
	\begin{enumerate}[\rm i.]
		\item\label{i.action_product}
		there exists an isomorphism \(\Psi:F*_{\theta}G\xrightarrow{\sim}H\) such that
		\(\Psi(F)=\Gamma_F\) and \(\Psi(G)=\Gamma_G\);
		
		\item\label{ii.action_product}
		the action of \(\Gamma_F\) on \(\T\) is semi-conjugate to the action of \(F\) on \(\T\)
		by a semi-conjugacy \(h_F\);
		
		\item\label{iii.action_product}
		the action of \(\Gamma_G\) on \(\T\) is semi-conjugate to the action of \(G\) on \(\T\)
		by a semi-conjugacy \(h_G\);
		
		\item\label{iv.action_product}
		\(\core(h_F)\subset \bigcup_{i=1}^n h_G^{-1}(y_i)\) and
		\(\core(h_F)\cap h_G^{-1}(y_i)\neq\emptyset\) for every \(i\in\{1,\ldots,n\}\);
		
		\item\label{v.action_product}
		\(\core(h_G)\subset \bigcup_{i=1}^n h_F^{-1}(x_i)\) and
		\(\core(h_G)\cap h_F^{-1}(x_i)\neq\emptyset\) for every \(i\in\{1,\ldots,n\}\);
		
		\item\label{vi.action_product}
		\((h_F^{-1}(x_1),h_G^{-1}(y_{\sigma(1)}),\ldots,h_F^{-1}(x_n),h_G^{-1}(y_{\sigma(n)}))\)
		is an ordered partition of \(\T\); moreover,
		\[
		\left(\bigcup_{i=1}^n h_F^{-1}(x_i)^\circ,\ \bigcup_{i=1}^n h_G^{-1}(y_{\sigma(i)})^\circ\right)
		\]
		is a proper ping-pong partition for \(\Gamma_F\) and \(\Gamma_G\).
	\end{enumerate}
	
	When such a subgroup \(H\) exists, we write
	\[
	H=(F,\bar x)\star_{\theta,\sigma}(G,\bar y).
	\]
	If \(\bar x=(x)\) and \(\bar y=(y)\) are singletons, we omit the trivial permutation and write
	\[
	H=(F, x)\star_{\theta}(G, y).
	\]
	And, if the stabilizers \(S_F\) and \(S_G\) are trivial, we omit the trivial isomorphism and write
	\[
	H=(F, x)\star_{\sigma}(G, y).
	\]
	
	\begin{remark*}
		Since \(\core(h_F)\) and \(\core(h_G)\) have no isolated points, items \ref{iii.action_product} and
		\ref{iv.action_product} imply that \(h_F^{-1}(x_i)\) and \(h_G^{-1}(y_i)\) are non-degenerate
		intervals for every \(i\in\{1,\ldots,n\}\). Moreover, by definition of the core,
		\[\partial \left(h_F^{-1}(x_i)\right)\subset \core(h_F)\subset \bigcup_{j=1}^n h_G^{-1}(y_j),\quad \text{ and similarly }\quad
		\partial \left(h_G^{-1}(y_i)\right)\subset \core(h_G)\subset \bigcup_{j=1}^n h_F^{-1}(x_j).\]
		Therefore, the requirement in item \ref{vi.action_product} is rather mild: if the sets
		\((h_F^{-1}(x_1),h_G^{-1}(y_{\sigma(1)}),\ldots,h_F^{-1}(x_n),h_G^{-1}(y_{\sigma(n)}))\)
		have pairwise disjoint interiors and are cyclically ordered, then their boundary points match and
		they form an ordered partition of \(\T\).
	\end{remark*}

	\subsection{Construction of the amalgamated product of circle actions}
	
	We now describe the construction of the amalgamated product of two circle actions.
	The following theorem states that, under the additional index assumption, this construction
	produces a subgroup of \(\homeo_+(\T)\) acting minimally on the circle and realizing the amalgamated product introduced above, and that the resulting action is unique up to conjugacy.
	
	\begin{theorem}\label{t.action_product}
		Let \(F,G\leq \homeo_+(\T)\) be countable subgroups, and let
		\(\bar x=(x_1,\ldots,x_n)\) and \(\bar y=(y_1,\ldots,y_n)\) be cyclically ordered
		\(n\)-tuples of points in \(\T\). Assume that, for all distinct \(i,j\in\{1,\ldots,n\}\),
		\begin{itemize}
			\item \(x_i\notin F.x_j\) and \(y_i\notin G.y_j\);
			\item \(\stab_F(x_i)=S_F\) and \(\stab_G(y_i)=S_G\), with
			\(\theta:S_F\xrightarrow{\sim}S_G\);
			\item \(S_F\lneqq F\), \(S_G\lneqq G\), and at least one of the indices
			\([F:S_F]\) and \([G:S_G]\) is greater than \(2\).
		\end{itemize}
		Let \(\sigma:\{1,\ldots,n\}\to\{1,\ldots,n\}\) be a cyclic permutation. Then there exists a subgroup \(H\) of \(\homeo_+(\T)\) acting minimally on \(\T\) such that
		\[
		H=(F,\bar x)\star_{\theta,\sigma}(G,\bar y).
		\]
		Moreover, \(H\) is unique up to conjugacy in \(\homeo_+(\T)\).
	\end{theorem}
	
	\begin{proof}~
		
		We begin by blowing up the actions of \(F\) and \(G\) along the orbits of
		\(x_1,\ldots,x_n\) and \(y_1,\ldots,y_n\), respectively, in such a way that the intervals
		inserted for the action of \(F\) are complementary to those inserted for the action of \(G\).
		This produces two semi-conjugacies \(h_F,h_G:\T\to\T\), together with the basic pieces on
		which the new actions will first be defined. The actions of \(F\) and \(G\) will then be
		constructed inductively on these pieces and extended to the whole circle.
		
		Choose a partition of the circle \((a_1,b_1,\ldots,a_n,b_n)\) with
		\(a_1<b_1<a_2<b_2<\cdots<a_n<b_n<a_1\).
		We require that \([a_i,b_i]\) be the preimage of \(x_i\) under the blow-up of \(F\), while
		\([b_i,a_{i+1}]\) is the preimage of \(y_{\sigma(i)}\) under the blow-up of \(G\).
		To this end, choose continuous weakly monotone maps \(h_F,h_G:\T\to\T\) such that:
		
		\begin{itemize}
			\item $h_F$ and $h_G$ are weakly monotone increasing and continuous functions,
			\item $h_F^{-1}(x_i)=[a_i,b_i]$ and $I_{\xi}:=h_F^{-1}(\xi)$ is a nontrivial closed interval for any $\xi\in \bigcup_{i=1}^nF.x_i$,
			\item $h_G^{-1}(y_{\sigma(i)})=[b_i,a_{i+1}]$ and $J_{\eta}:=h_G^{-1}(\eta)$ is a nontrivial closed interval for any $\eta\in \bigcup_{i=1}^nG.y_i$,
			\item $h_F^{-1}(z)$ is a point for any $z\notin \bigcup_{i=1}^n F.x_i$,
			\item $h_G^{-1}(z)$ is a point for any $z\notin \bigcup_{i=1}^n G.y_i$.
		\end{itemize}

		\bigskip
		
		\noindent\textbf{Auxiliary sets and first partial maps.}
		
		Set		\[
		\cI=\bigcup_{i=1}^n\bigcup_{\xi\in F.x_i} I_\xi,
		\qquad
		\cJ=\bigcup_{i=1}^n\bigcup_{\eta\in G.y_i} J_\eta.
		\]
		For each \(i\in\{1,\ldots,n\}\) and each \(\xi\in F.x_i\), let
		\(A_\xi:I_{x_i}\to I_\xi\) be the unique orientation-preserving affine homeomorphism.
		Similarly, for each \(\eta\in G.y_i\), let \(B_\eta:J_{y_i}\to J_\eta\) be the unique
		orientation-preserving affine homeomorphism. Define inductively
		\begin{align*}
			K_F^m&=
			\begin{cases}
				\overline{\T\smallsetminus\cI}, & \text{if }m=0,\\[0.4em]
				\displaystyle \bigcup_{i=1}^n\bigcup_{\xi\in F.x_i}
				A_\xi(K_G^{m-1}\cap I_{x_i}), & \text{if }m\geq 1,
			\end{cases}\\[0.8em]
			K_G^m&=
			\begin{cases}
				\overline{\T\smallsetminus\cJ}, & \text{if }m=0,\\[0.4em]
				\displaystyle \bigcup_{i=1}^n\bigcup_{\eta\in G.y_i}
				B_\eta(K_F^{m-1}\cap J_{y_i}), & \text{if }m\geq 1.
			\end{cases}
		\end{align*}
		Set \(K^m:=K_F^m\cup K_G^m\) for each \(m\in\NN\), and
		\[
		K:=\bigcup_{m\in\NN} K^m.
		\]
		By construction,
		\[
		\bigcup_{m\in\NN}K_F^m
		=
		K\cap \bigcup_{i=1}^n J_{y_i},
		\qquad
		\bigcup_{m\in\NN}K_G^m
		=
		K\cap \bigcup_{i=1}^n I_{x_i}.
		\]
		For each \(i\in\{1,\ldots,n\}\), each \(\xi\in F.x_i\), and each \(\eta\in G.y_i\), choose
		elements \(f_\xi\in F\) and \(g_\eta\in G\) such that
		\[
		f_\xi(x_i)=\xi,
		\qquad
		g_\eta(y_i)=\eta,
		\]
		and define
		\[
		\Psi(f_\xi)\big|_{K\cap I_{x_i}}:=A_\xi,
		\qquad
		\Psi(g_\eta)\big|_{K\cap J_{y_i}}:=B_\eta.
		\]
		We will first define an action of \(F*_{S}G\) on the subset \(K\), and later extend it to
		the whole circle \(\T\).
		
		\bigskip
		
		\noindent\textbf{Definition of the action of the stabilizers.}
		
		By the universal property of the amalgamated free product, it is enough to define
		compatible actions of \(F\) and \(G\) on \(K\), that is, actions whose restrictions to
		\(S_F\) and \(S_G\) agree under the chosen isomorphism \(\theta:S_F\to S_G\). We begin with
		the stabilizers.
		
		The tautological actions of \(F\) and \(G\) induce actions
		\[
		\psi_F:F\to \homeo_+(\T\smallsetminus\cI),
		\qquad
		\psi_G:G\to \homeo_+(\T\smallsetminus\cJ),
		\]
		defined by
		\[
		\psi_F(f)(z)=h_F^{-1}fh_F(z)
		\quad\text{for }f\in F,\ z\in \T\smallsetminus\cI,
		\]
		and
		\[
		\psi_G(g)(z)=h_G^{-1}gh_G(z)
		\quad\text{for }g\in G,\ z\in \T\smallsetminus\cJ.
		\]
		By continuity, these actions extend to \(K_F^0\) and \(K_G^0\), and we still denote the
		extensions by \(\psi_F\) and \(\psi_G\), respectively.
		
		Now let \(s\in S_F\). Since \(S_F\) and \(S_G\) are isomorphic, fix an isomorphism
		\[
		\theta:S_F\xrightarrow{\sim} S_G.
		\]
		For \(s\in S_F\), define
		\[
		\Psi(s)\big|_{K_F^0}:=\psi_F(s),
		\qquad
		\Psi(s)\big|_{K_G^0}:=\psi_G(\theta(s)).
		\]
		Take \(s\in S_F\), and let \(\xi_1,\xi_2\in F.x_i\) and \(\eta_1,\eta_2\in G.y_i\) satisfy
		\(s(\xi_1)=\xi_2\) and \(\theta(s)(\eta_1)=\eta_2\). We then define, for each
		\(i\in\{1,\ldots,n\}\) and each \(m\in\NN\),
		\begin{align*}
			\Psi(s)\big|_{A_{\xi_1}(K_G^m\cap I_{x_i})}
			&:=
			A_{\xi_2}\,\Psi(f_{\xi_2}^{-1}sf_{\xi_1})\big|_{K_G^m}\,A_{\xi_1}^{-1},\\
			\Psi(s)\big|_{B_{\eta_1}(K_F^m\cap J_{y_i})}
			&:=
			B_{\eta_2}\,\Psi\!\left(\theta^{-1}(g_{\eta_2}^{-1}\theta(s)g_{\eta_1})\right)\big|_{K_F^m}\,B_{\eta_1}^{-1}.
		\end{align*}
		Thus, for every \(s\in S_F\), the map \(\Psi(s)\) is a well-defined homeomorphism of \(K\)
		preserving the orientation induced from \(\T\). For \(s\in S_G\), we set
		\[
		\Psi(s):=\Psi(\theta^{-1}(s)).
		\]
		In particular, the two stabilizer actions coincide, and hence
		\[
		\Psi(S_F)=\Psi(S_G).
		\]
		
		\bigskip
		
		\noindent\textbf{Definition of the actions of \(F\) and \(G\).} 
		
		We now define the actions of \(F\) and \(G\) on \(K\). Recall that, for each \(i\in\{1,\ldots,n\}\), \(\xi\in F.x_i\), and \(\eta\in G.y_i\), we have already fixed elements \(f_\xi\in F\) and \(g_\eta\in G\) such that \(f_\xi(x_i)=\xi\) and \(g_\eta(y_i)=\eta\), and defined
		 \[ \Psi(f_\xi)\big|_{K\cap I_{x_i}}:=A_\xi, \qquad \Psi(g_\eta)\big|_{K\cap J_{y_i}}:=B_\eta. \]
		Moreover, for every \(f\in F\), the map \(\Psi(f)\big|_{K_F^0}\) is already defined by \(\psi_F(f)\), and similarly, for every \(g\in G\), the map \(\Psi(g)\big|_{K_G^0}\) is already defined by \(\psi_G(g)\). 
		
		Let \(f\in F\). If \(f(x_i)=\xi\), then \(f_\xi^{-1}f\in S_F\). Since \(\Psi(f_\xi^{-1}f)\) is already defined on \(K\), and \(\Psi(f_\xi)\) is already defined on \(K\cap I_{x_i}\), the only compatible  definition is 
		\[ \Psi(f)\big|_{K\cap I_{x_i}} := \Psi(f_\xi)\,\Psi(f_\xi^{-1}f). \] 
		This defines \(\Psi(f)\) on \(K\cap I_{x_i}\) for every \(i\in\{1,\ldots,n\}\). It remains to define \(\Psi(f)\) on the points of \(K\) outside 
		\[ \left(K\cap \bigcup_{i=1}^n I_{x_i}\right)\cup K_F^0. \] 
		By construction, 
		\[ K\smallsetminus \left(\left(K\cap \bigcup_{i=1}^n I_{x_i}\right)\cup K_F^0\right) = \left(K\cap \bigcup_{i=1}^n J_{y_i}\right)\smallsetminus K_F^0 = \bigcup_{m\geq 1} K_F^m. \] 
		Now let \(m\geq 1\) and \(z\in K_F^m\). Then there exists \(\xi'\in F.x_i\) such that 
		\[ A_{\xi'}^{-1}(z)\in K_G^{m-1}\subset K\cap \bigcup_{i=1}^n I_{x_i}. \] 
		Since \(\Psi(f\,f_{\xi'})\) is already defined at \(A_{\xi'}^{-1}(z)\), we set 
		\[ \Psi(f)(z) := \Psi(f\,f_{\xi'})(A_{\xi'}^{-1}(z)). \]
		Equivalently, 
		\[ \Psi(f)(z) =\Psi(f)\,\Psi(f_{\xi'})(A_{\xi'}^{-1}(z))= \Psi(f)\,A_{\xi'}\,A_{\xi'}^{-1}(z). \] 
		In this way, \(\Psi(f)\) is defined on all of \(K\), and \(\Psi:F\times K\to K\) is a
		continuous group action by orientation-preserving homeomorphisms, which extends to a
		continuous action on \(\overline K\). The construction for \(G\) is entirely analogous, and yields a continuous action 
		\[ \Psi:G\times \overline K\to \overline K. \]
		
		Since \(\overline K\subset \T\) is closed, the actions of \(F\) and \(G\) on \(\overline K\)
		extend to continuous orientation-preserving actions on \(\T\), which we still denote by
		\[
		\Psi:F\times \T\to\T,
		\qquad
		\Psi:G\times \T\to\T.
		\]
		
		\bigskip
		
		\smallskip
		
		\noindent\textbf{Verification of the defining properties.}
		
		Let
		\[
		H:=\langle \Psi(F),\Psi(G)\rangle\leq \homeo_+(\T),
		\qquad
		\Gamma_F:=\Psi(F),
		\qquad
		\Gamma_G:=\Psi(G).
		\]
		By construction, the action of \(\Gamma_F\) on \(\T\) is semi-conjugate to the action of
		\(F\) on \(\T\) by the semi-conjugacy \(h_F\), and similarly the action of \(\Gamma_G\) on
		\(\T\) is semi-conjugate to the action of \(G\) on \(\T\) by the semi-conjugacy \(h_G\).
		Moreover,
		\[
		\core(h_F)=\T\smallsetminus \cI^\circ \subset \bigcup_{i=1}^n [b_i,a_{i+1}]
		= \bigcup_{i=1}^n h_G^{-1}(y_i),
		\]
		and
		\[
		\core(h_G)=\T\smallsetminus \cJ^\circ \subset \bigcup_{i=1}^n [a_i,b_i]
		= \bigcup_{i=1}^n h_F^{-1}(x_i).
		\]
		Since every interval \(h_G^{-1}(y_i)\) and \(h_F^{-1}(x_i)\) has endpoints in
		\(\{a_1,b_1,\ldots,a_n,b_n\}\subset \core(h_F)\cap \core(h_G)\), it follows that, for every
		\(i\in\{1,\ldots,n\}\),
		\[
		\core(h_F)\cap h_G^{-1}(y_i)\neq\emptyset
		\quad\text{and}\quad
		\core(h_G)\cap h_F^{-1}(x_i)\neq\emptyset.
		\]
		Thus items \ref{ii.action_product}, \ref{iii.action_product}, \ref{iv.action_product}, and
		\ref{v.action_product} of Definition~\ref{d.action_product} are satisfied.
		
		It remains to verify items \ref{i.action_product} and \ref{vi.action_product}.
		By construction,
		\[
		h_F^{-1}(x_i)=[a_i,b_i],
		\qquad
		h_G^{-1}(y_{\sigma(i)})=[b_i,a_{i+1}],
		\]
		and therefore
		\[
		\bigl(h_F^{-1}(x_1),h_G^{-1}(y_{\sigma(1)}),\ldots,
		h_F^{-1}(x_n),h_G^{-1}(y_{\sigma(n)})\bigr)
		\]
		is an ordered partition of \(\T\).
		
		Recall that \(I_{x_i}\) and \(J_{y_{\sigma(i)}}\) denote the intervals
		\(h_F^{-1}(x_i)\) and \(h_G^{-1}(y_{\sigma(i)})\), respectively. Let \(f\in F\smallsetminus S_F\) and \(i\in\{1,\ldots,n\}\). Then there exists
		\(\xi\in F.x_i\), with \(\xi\notin\{x_1,\ldots,x_n\}\), such that
		\(f_\xi^{-1}f\in S_F\). Hence
		\[
		\Psi(f)(I_{x_i})
		=
		\Psi(f_\xi)\Psi(f_\xi^{-1}f)(I_{x_i})
		=
		\Psi(f_\xi)(I_{x_i})
		=
		I_\xi.
		\]
		Since \(I_\xi\cap I_{x_j}=\emptyset\) for every \(j\in\{1,\ldots,n\}\), it follows that
		\(I_\xi\) is contained in the interior of one of the intervals \(J_{y_{\sigma(j)}}\). Thus
		\[
		\Psi(f)\left(\bigcup_{i=1}^n I_{x_i}\right)
		\subset
		\bigcup_{i=1}^n J^\circ_{y_{\sigma(i)}}.
		\]
		Similarly, for every \(g\in G\smallsetminus S_G\),
		\[
		\Psi(g)\left(\bigcup_{i=1}^n J_{y_{\sigma(i)}}\right)
		\subset
		\bigcup_{i=1}^n I^\circ_{x_i}.
		\]
		Therefore
		\[
		\left(\bigcup_{i=1}^n I_{x_i},\ \bigcup_{i=1}^n J_{y_{\sigma(i)}}\right)
		\]
		is a proper ping-pong partition for the actions of \(\Gamma_F\) and \(\Gamma_G\).
		
		Finally, if we denote
		\[
		S:=\Psi(S_F)=\Psi(S_G),
		\]
		then the ping-pong lemma yields
		\[
		H=\langle \Psi(F),\Psi(G)\rangle=\Psi(F)*_S\Psi(G)\cong F*_\theta G.
		\]
		This proves items \ref{i.action_product} and \ref{vi.action_product}.
		
		Since \(\overline K\subset \T\) is a nonempty closed \(H\)-invariant subset, the action of
		\(H\) on \(\T\) admits a minimal set \(M\subset \overline K\). By
		Lemma~\ref{l.unique_minimal_ping-pong}, this minimal set is infinite. Hence \(M\) is either
		all of \(\T\) or a Cantor set; see \cite[Proposition~5.6]{ghys-circle}. In the latter case,
		collapsing each connected component of \(\T\smallsetminus M\) to a point yields a minimal
		action on \(\T\) which is semi-conjugate to the original one. Replacing \(H\) by this new
		action, we may therefore assume that the resulting subgroup \(H\) acts minimally on \(\T\).
		Moreover, the verifications of items \ref{i.action_product}--\ref{vi.action_product} are
		unchanged by this collapse.
		
		The uniqueness statement will be proved in Lemma~\ref{l.product_unicity}. This completes the
		proof of the theorem.
	\end{proof}
	
	\begin{remark}
		The above construction depends on the identification of the stabilizers
		\(S_F\) and \(S_G\). In particular, replacing the convention
		\(\Psi(s)=\Psi(\theta^{-1}(s))\) on \(S_G\) by
		\(\Psi(s)=\Psi(\theta^{-1}(s^{-1}))\) may lead to a different dynamical type
		for the resulting action. This choice will be relevant later in determining whether the resulting action is M\"obius-like.
	\end{remark}
	
	\subsection{Uniqueness of the amalgamated product of circle actions}
	We now show that the amalgamated product of circle actions is well defined: any two realizations with the same data are conjugate.
	
	\begin{lemma}\label{l.product_unicity}
		Let \(H_1,H_2\leq \homeo_+(\T)\) be two amalgamated products of \(F\) and \(G\) on
		\(\{x_1,\ldots,x_n\}\) and \(\{y_1,\ldots,y_n\}\), with respect to the same isomorphism
		\(\theta:S_F\to S_G\) and the same permutation \(\sigma\). Then the actions of \(H_1\) and
		\(H_2\) on \(\T\) are conjugate.
	\end{lemma}
	\begin{proof}
		For \(q\in\{1,2\}\), let \(\Gamma_F^q,\Gamma_G^q\leq H_q\), the semi-conjugacies
		\(h_{F_q},h_{G_q}:\T\to\T\), and the isomorphism \(\Psi_q:F*_\theta G\xrightarrow{\sim} H_q\) be as in Definition~\ref{d.action_product}. Also set
		\[
		X_F^q:=\core(h_{F_q}),
		\qquad
		X_G^q:=\core(h_{G_q}).
		\]
		Let \(\widetilde X_F^q\subset X_F^q\) and \(\widetilde X_G^q\subset X_G^q\) denote the subsets
		where \(h_{F_q}\) and \(h_{G_q}\) are injective, respectively.
		
		By the defining properties of the semi-conjugacies, for every \(f\in F\), \(g\in G\), one has
		\[
		h_{F_q}\circ \Psi_q(f)=f\circ h_{F_q},
		\qquad
		h_{G_q}\circ \Psi_q(g)=g\circ h_{G_q},
		\]
		and therefore, for \(x\in \widetilde X_F^q\) and \(y\in \widetilde X_G^q\),
		\[
		\Psi_q(f)(x)=h_{F_q}^{-1}fh_{F_q}(x),
		\qquad
		\Psi_q(g)(y)=h_{G_q}^{-1}gh_{G_q}(y).
		\]
		
		We may thus define maps
		\[
		\alpha_F:\widetilde X_F^1\to \widetilde X_F^2,
		\qquad
		\alpha_G:\widetilde X_G^1\to \widetilde X_G^2
		\]
		by
		\[
		\alpha_F(x):=h_{F_2}^{-1}h_{F_1}(x),
		\qquad
		\alpha_G(x):=h_{G_2}^{-1}h_{G_1}(x).
		\]
		Since \(h_{F_q}\) and \(h_{G_q}\) are monotone semi-conjugacies, both \(\alpha_F\) and
		\(\alpha_G\) are continuous and order-preserving, and extend continuously to \(X_F^1\) and
		\(X_G^1\), respectively.
		
		We claim that these extensions agree on \(X_F^1\cap X_G^1\). Indeed, observe that the subsets \(\widetilde X_F^q\) and \(\widetilde X_G^q\) are disjoint. Now, suppose that \(x\in X_F^1\cap X_G^1\). By items \ref{iv.action_product} and \ref{v.action_product} of
		Definition~\ref{d.action_product}, there exist \(i,j\in\{1,\ldots,n\}\) such that
		\[
		x\in h_{F_1}^{-1}(x_i)\cap h_{G_1}^{-1}(y_j).
		\]
		Since this intersection has empty interior, the point \(x\) is a limit point of both
		intervals \(h_{F_1}^{-1}(x_i)\) and \(h_{G_1}^{-1}(y_j)\). Moreover, by
		item~\ref{vi.action_product} of Definition~\ref{d.action_product}, one has
		\[
		j=\sigma(i)\qquad\text{or}\qquad j=\sigma(i-1).
		\]
		Assume, without loss of generality, that \(j=\sigma(i)\). Then
		\[
		x=\sup h_{F_1}^{-1}(x_i)=\inf h_{G_1}^{-1}(y_{\sigma(i)}).
		\]
		and therefore
		\[
		\alpha_F(x)=\sup h_{F_2}^{-1}(x_i)
		=\inf h_{G_2}^{-1}(y_{\sigma(i)})
		=\alpha_G(x).
		\]
		The other case is analogous. Hence \(\alpha_F\) and \(\alpha_G\) glue together to define a
		continuous order-preserving map
		\[
		\alpha:X_F^1\cup X_G^1\longrightarrow X_F^2\cup X_G^2.
		\]
		
		\smallskip
		\noindent\textbf{Compatibility with the stabilizers.} 
		
		Set \(S_q:=\Gamma_F^q\cap \Gamma_G^q \), for \(q\in\{1,2\}\). We claim that, for every \(s\in S_1\), there exists \(s'\in S_2\) such that 
		\[ \alpha\circ s=s'\circ \alpha \qquad\text{on}\quad X_F^1\cup X_G^1. \] 
		Fix \(s\in S_1\). Since \(s\in \Gamma_F^1\), there exists \(s_F\in S_F\) such that \(\Psi_1(s_F)=s\).
		
		Let \(s':=\Psi_2(s_F)\in S_2\). Then, for every \(x\in \widetilde X_F^1\), the defining property of the semi-conjugacies gives 
		\[ s(x)=h_{F_1}^{-1}s_Fh_{F_1}(x), \qquad s'(\alpha(x))=h_{F_2}^{-1}s_Fh_{F_2}(\alpha(x)). \] 
		Since \(\alpha(x)=h_{F_2}^{-1}h_{F_1}(x)\), it follows that 
		\[ \alpha\circ s(x)=s'\circ \alpha(x) \qquad\text{for every }x\in \widetilde X_F^1. \]
		On the other hand, since \(H_1\) and \(H_2\) are amalgamated products of circle actions with respect to the same isomorphism \(\theta:S_F\to S_G\), the element \(s\in S_1\) corresponds on the \(G\)-side to \(\theta(s_F)\in S_G\), and the same holds for \(s'\in S_2\). Therefore, for every \(x\in \widetilde X_G^1\), 
		\[ s(x)=h_{G_1}^{-1}\theta(s_F)h_{G_1}(x), \qquad s'(\alpha(x))=h_{G_2}^{-1}\theta(s_F)h_{G_2}(\alpha(x)), \] 
		and hence 
		\[ \alpha\circ s(x)=s'\circ \alpha(x) \qquad\text{for every }x\in \widetilde X_G^1. \] 
		By continuity, the same identity holds on \(X_F^1\cup X_G^1\). Thus, for every \(s\in S_1\), there exists \(s'\in S_2\) such that 
		\[ \alpha\circ s=s'\circ \alpha \qquad\text{on}\quad X_F^1\cup X_G^1. \]
		 
		\smallskip 
		
		\noindent\textbf{Inductive extension to the whole circle.} 
		
		We now extend \(\alpha\) inductively to the whole circle. For each \(i\in\{1,\ldots,n\}\) and each \(\xi\in F.x_i\), choose \(f_\xi\in F\) such that \(f_\xi(x_i)=\xi\). Similarly, for each \(\eta\in G.y_i\), choose \(g_\eta\in G\) such that \(g_\eta(y_i)=\eta\). Also set 
		\[ I_{x_i}^q:=h_{F_q}^{-1}(x_i), \qquad J_{y_i}^q:=h_{G_q}^{-1}(y_i), \quad \text{ for } q\in\{1,2\}. \]
		Define inductively 
		\begin{align*} Y_F^m&= \begin{cases} X_F^1, & \text{if }m=0,\\[0.4em] \displaystyle \bigcup_{i=1}^n \bigcup_{\xi\in F.x_i} \Psi_1(f_\xi)\bigl(Y_G^{m-1}\cap I_{x_i}^1\bigr), & \text{if }m\ge 1, \end{cases}\\[0.8em] Y_G^m&= \begin{cases} X_G^1, & \text{if }m=0,\\[0.4em] \displaystyle \bigcup_{i=1}^n \bigcup_{\eta\in G.y_i} \Psi_1(g_\eta)\bigl(Y_F^{m-1}\cap J_{y_i}^1\bigr), & \text{if }m\ge 1, \end{cases} 
		\end{align*} 
		and 
		\begin{align*} Z_F^m&= \begin{cases} X_F^2, & \text{if }m=0,\\[0.4em] \displaystyle \bigcup_{i=1}^n \bigcup_{\xi\in F.x_i} \Psi_2(f_\xi)\bigl(Z_G^{m-1}\cap I_{x_i}^2\bigr), & \text{if }m\ge 1, \end{cases}\\[0.8em] Z_G^m&= \begin{cases} X_G^2, & \text{if }m=0,\\[0.4em] \displaystyle \bigcup_{i=1}^n \bigcup_{\eta\in G.y_i} \Psi_2(g_\eta)\bigl(Z_F^{m-1}\cap J_{y_i}^2\bigr), & \text{if }m\ge 1. \end{cases} 
		\end{align*} 
		We then define \(\alpha\) inductively by 
		\begin{align*} \alpha\big|_{\Psi_1(f_\xi)(Y_G^m\cap I_{x_i}^1)} &:=\Psi_2(f_\xi)\,\alpha\big|_{Y_G^m\cap I_{x_i}^1}\,\Psi_1(f_\xi)^{-1},\\ \alpha\big|_{\Psi_1(g_\eta)(Y_F^m\cap J_{y_i}^1)} &:=\Psi_2(g_\eta)\,\alpha\big|_{Y_F^m\cap J_{y_i}^1}\,\Psi_1(g_\eta)^{-1}. 
		\end{align*} 
		Set 
		\[ Y:=\bigcup_{m\ge 0}(Y_F^m\cup Y_G^m), \qquad Z:=\bigcup_{m\ge 0}(Z_F^m\cup Z_G^m). \] 
		By induction, \(\alpha\) extends to a well-defined continuous monotone map of degree one 
		\[ \alpha:\overline Y\longrightarrow \overline Z. \] 
		Moreover, \(\overline Y\) and \(\overline Z\) are closed invariant subsets for the actions of \(H_1\) and \(H_2\), respectively. Since both actions are minimal, it follows that 
		\[ \overline Y=\overline Z=\T. \]
		Hence \(\alpha\) extends to a continuous monotone degree-one map 
		\[ \alpha:\T\to\T. \] 
		Finally, by construction, for every \(\gamma_1\in \Gamma_F^1\cup \Gamma_G^1\) there exists \(\gamma_2\in \Gamma_F^2\cup \Gamma_G^2\) such that 
		\[ \alpha\circ \gamma_1=\gamma_2\circ \alpha \qquad\text{on }\T. \] 
		Since 
		\[ H_1=\Gamma_F^1*_{S_1}\Gamma_G^1, \qquad H_2=\Gamma_F^2*_{S_2}\Gamma_G^2, \] 
		it follows that \(\alpha\) conjugates the actions of \(H_1\) and \(H_2\) on \(\T\).
	\end{proof}
	
	
	\section{Tracking the number of fixed points}\label{s.tracking}
	
	
	In this section we study how the number of fixed points behaves under the amalgamated
	product construction. Our goal is to produce actions on \(\T\) with at most \(2n\) fixed
	points. We begin with a general mechanism yielding such examples, and then derive
	conditions ensuring that the resulting actions are not conjugate into any subgroup of
	\(\psl^{(k)}(2,\RR)\).
	
	\subsection{Constructing group actions with at most \(2n\) fixed points} The main result of this subsection is Theorem~\ref{t.main_examples_2n}, which provides a general mechanism for constructing group actions on \(\T\) with at most \(2n\) fixed points. Its proof is based on Lemma~\ref{l.fixed_ping-pong} below.
	
	For a nonempty subset \(U\) of a topological space, we denote by \(b_0(U)\) the number of
	connected components of \(U\).
	
	\begin{lemma}\label{l.fixed_ping-pong}
		Let \(F,G\leq \homeo_+(\T)\), set \(H:=\langle F,G\rangle\) and \(S:=F\cap G\), and let
		\((U,V)\) be a proper ping-pong partition for \(F\) and \(G\). Assume that the action of
		\(H\) on \(\T\) is minimal. Then, for every \(h\in H\) which is not conjugate in \(H\) to
		an element of \(F\cup G\), one has
		\[
		\#\fix(h)\le 2\min\{b_0(U),b_0(V)\}.
		\]
	\end{lemma}
	
	\begin{proof}
		Set \(F^*:=F\smallsetminus S\) and \(G^*:=G\smallsetminus S\). Since \(h\in H\) is not conjugate in \(H\) to an element of \(F\cup G\), and since \(H\cong F*_S G\),
		either \(h\) or \(h^{-1}\) is conjugate in \(H\) to an element \(\tilde h\) of the form
		\[
		\tilde h=g_k f_k\cdots g_1 f_1,
		\qquad
		f_i\in F^*,\ g_i\in G^*.
		\]
		In particular, \(\tilde h\) is represented by a cyclically reduced alternating word.
		Moreover, conjugation preserves the number of fixed points, and
		\(\#\fix(h)=\#\fix(h^{-1})\). Therefore, \(\#\fix(h)=\#\fix(\tilde h)\).
		
		By the ping-pong property, one has \(\tilde h(U)\subset U\). We claim that this
		inclusion is proper. Indeed, if \(\tilde h(U)=U\), then, since the action is minimal,
		Lemma~\ref{l.minimal_ping-pong} gives \(\T= \overline{ U \cup V}\). Therefore \(\tilde h(V)=V\) as well. But then \(\tilde h\) preserves both sides of the ping-pong partition, which by the normal-form consequence of the ping-pong lemma forces \(\tilde h\in S\), a contradiction.
		
		Let \(I\) be a connected component of \(U\). Then either \(\fix(\tilde h)\cap I=\emptyset\), or the nested family \(\{\tilde h^n(\overline I)\}_{n\in\NN}\) defines a nonempty attracting interval 
		\[ A:=\bigcap_{n\in\NN}\tilde h^n(\overline I), \] 
		which may a priori be degenerate. We claim that \(A\) is always reduced to a single point. To see this, consider the following family of open sets, analogous to the one used in the proof of Lemma~\ref{l.minimal_ping-pong}:
		\begin{align*} 
			\Omega_0 &= U\cup V,\\
			\Omega_1 &= G^*(V)\cup F^*(U),\\
			\Omega_2 &= G^*F^*(U)\cup F^*G^*(V),\\
			\Omega_3 &= G^*F^*G^*(V)\cup F^*G^*F^*(U),\ldots 
		\end{align*} 
		Since \(I\subset U\), one has \(\tilde h^n(I)\subset \Omega_{2kn}\) for every \(n\in\NN\). As
		\((\Omega_n)\) is decreasing,  it follows that \(\tilde h^n(I)\subset \Omega_i\) for every
		\(i\le 2kn\). Hence every point of \(A^\circ\) belongs to \(\Omega_i\) for every \(i\), so
		\[
		A^\circ\subset \bigcap_{n\in\NN}\Omega_n.
		\]		
		Moreover, for every \(f\in F\), \(g\in G\), and every \(n\ge0\), one has 
		\[ f(\Omega_{n+1})\subset \Omega_n, \qquad g(\Omega_{n+1})\subset \Omega_n. \] 
		The set \(\bigcap_{n\in\NN}\Omega_n\) is \(H\)-invariant, and its complement is nonempty,
		since \(\partial U\cap \Omega_0=\emptyset\). As the action of \(H\) is minimal, this complement must be dense. It follows that \(A\) has empty interior, and therefore \(A\) is reduced to a single point.
		
		We conclude that, for each connected component \(I\) of \(U\), either
		\(\fix(\tilde h)\cap I=\emptyset\), or the sequence \(\tilde h^n(\overline I)\) converges to
		a point of \(\overline I\), which is the unique fixed point of \(\tilde h\) in \(\overline I\).
		If this point lies in the interior of \(I\), then it is an attracting fixed point of
		\(\tilde h\).
		
		Applying the same argument to \(\tilde h^{-1}\) and the connected components \(J\) of \(V\),
		we obtain that, for each such \(J\), either \(\fix(\tilde h)\cap J=\emptyset\), or the
		sequence \(\tilde h^{-n}(\overline J)\) converges to a point of \(\overline J\), which is
		the unique fixed point of \(\tilde h\) in \(\overline J\). If this point lies in the
		interior of \(J\), then it is a repelling fixed point of \(\tilde h\). Finally, if $\tilde{h}^n(\overline{I})$ converges to an endpoint of $I$ which is also an endpoint of a component $J$ of $V$, then $\tilde{h}$ has a parabolic fixed point which is the only fixed point of $\tilde{h}$ in $\overline{I}\cup \overline{J}$.
		
		We have shown that \(\tilde h\) has at most one fixed point in each closed component \(\overline I\) of \(\overline U\) and each closed component \(\overline J\) of \(\overline V\). In particular, 
		\[ \#\bigl(\fix(\tilde h)\cap \overline U\bigr)\leq b_0(U), \qquad \#\bigl(\fix(\tilde h)\cap \overline V\bigr)\leq b_0(V). \] 
		Since the action is minimal, Lemma~\ref{l.minimal_ping-pong} gives \(\overline{U\cup V}=\T\), and therefore 
		\[ \#\fix(h)=\#\fix(\tilde h)\leq b_0(U)+b_0(V). \] 
		
		We now replace the given ping-pong partition \((U,V)\) by a coarser one obtained as follows. Whenever two or more consecutive connected components belong to the same side of the partition, we replace the whole maximal block by the smallest open interval of \(\T\) containing it. Performing this simultaneously for \(U\) and \(V\), we obtain a new proper ping-pong partition \((\widetilde{U},\widetilde{V})\). By construction, the connected components of \(\widetilde{U}\) and \(\widetilde{V}\) alternate around the circle. In particular, 
		\[ b_0(\widetilde{U})=b_0(\widetilde{V}) \le \min\{b_0(U),b_0(V)\}. \] 
		Applying the previous argument to this coarser partition, we get 
		\[ \#\fix(h)\le b_0(\widetilde{U})+b_0(\widetilde{V}) \le 2\min\{b_0(U),b_0(V)\}. \]
		This proves the lemma.
	\end{proof}
	
	\begin{remark}
		Lemma~\ref{l.fixed_ping-pong} has concrete applications to finitely generated analytic
		subgroups of \(\mathrm{Diff}^\omega_+(\T)\). In particular, it can be used to show that a
		certain explicit subgroup of the amalgamated product
		\(\psl(2,\RR)*_{\SO(2)}\mathrm{SL}(2,\RR)\) acts minimally on \(\T\), every nontrivial element
		has at most two fixed points, and the action is not topologically conjugate to a
		subgroup of \(\psl(2,\RR)\). We postpone the explicit construction to a subsequent
		article.
	\end{remark}
	
	We now apply Lemma~\ref{l.fixed_ping-pong} to the amalgamated product construction of
	Section~\ref{s.amalgamated_product}. Under the hypotheses below, the resulting action still has at
	most \(2n\) fixed points.
	
	\begin{theorem}\label{t.main_examples_2n}
		Let \(F,G\leq \homeo_+(\T)\) be countable subgroups such that every nontrivial element has at most \(2n\) fixed points. Let
		\[
		\bar x=(x_1,\ldots,x_n),\qquad \bar y=(y_1,\ldots,y_n)
		\]
		be cyclically ordered \(n\)-tuples of points in \(\T\). Assume that, for all distinct
		\(i,j\in\{1,\ldots,n\}\),
		\begin{enumerate}[\rm i.]
			\item \(x_i\notin F.x_j\) and \(y_i\notin G.y_j\);
			\item \(\stab_F(x_i)=S_F\) and \(\stab_G(y_i)=S_G\), where
			\(\theta:S_F\xrightarrow{\sim}S_G\) is a fixed isomorphism;
			\item \(\fix(s_f)=\{x_1,\ldots,x_n\}\) and \(\fix(s_g)=\{y_1,\ldots,y_n\}\) for every nontrivial \(s_f\in S_F\) and \(s_g\in S_G\);
			\item \(S_F\lneqq F\), \(S_G\lneqq G\), and at least one of the indices
			\([F:S_F]\) and \([G:S_G]\) is greater than \(2\).
		\end{enumerate}
		Let \(\sigma\) be any cyclic order-preserving permutation of \(\{1,\ldots,n\}\). Then the
		amalgamated product
		\[
		(F,\bar x)\star_{\theta,\sigma}(G,\bar y)
		\]
		has at most \(2n\) fixed points.
	\end{theorem}
	
	\begin{proof}
		Using the notation of Definition~\ref{d.action_product}, observe first that any element of the amalgamated product which is conjugate into \(\Psi(F)\) or \(\Psi(G)\) has at most \(2n\) fixed points. Indeed, the only elements whose fixed-point set is modified by the blow-up are those in the stabilizers \(S_F\) and \(S_G\), and these fix precisely the sets \(\bar x\) and \(\bar y\), respectively. After the blow-up, they therefore have exactly \(2n\) fixed points. 
		
		Now \(\Psi(F)\) and \(\Psi(G)\) admit a proper ping-pong partition with \(n\) connected components on each side. Hence, by Lemma~\ref{l.fixed_ping-pong}, every element of the amalgamated product which is not conjugate into \(\Psi(F)\cup\Psi(G)\) also has at most \(2n\) fixed points. Therefore every nontrivial element of \[ (F,\bar x)\star_{\theta,\sigma}(G,\bar y) \] has at most \(2n\) fixed points.
	\end{proof}
	
	\subsection{Conditions for being M\"obius-like}
	
	The previous subsection gives a general mechanism for constructing actions with a uniform
	bound on the number of fixed points. We now turn to additional hypotheses under which the
	resulting amalgamated product is M\"obius-like. The key point is to control how the images
	of different components of the ping-pong partition are allowed to interlace on the circle.
	This is encoded in the following notion of properly unlinked points and intervals.
	
	\begin{definition}\label{d.unlinked} 
		Let \(G\subset \homeo_+(\T)\) be a subset, and let \(x,y\in \T\) be distinct. We say that the (partial) orbits of \(x\) and \(y\) are \emph{unlinked with respect to \(G\)} if, for every \(g\in G\), the pairs \(\{x,y\}\) and \(\{g(x),g(y)\}\) are unlinked, that is, if either
		\[ \{g(x),g(y)\}\subset [x,y] \qquad \text{or} \qquad \{g(x),g(y)\}\subset [y,x]. \] 
		A cyclically ordered \(n\)-tuple of points \((x_1,\ldots,x_n)\) is said to be \emph{properly unlinked with respect to \(G\)} if the following two conditions hold: 
		\begin{enumerate}[\rm i.] 
			\item for every pair \(\{x_i,x_j\}\), the (partial) orbits of \(x_i\) and \(x_j\) are unlinked with respect to \(G\); 
			\item for every \(g\in G\) such that \(\fix(g)\cap\{x_1,\ldots,x_n\}=\emptyset\), there exists a pair of consecutive points \(\{x_i,x_{i+1}\}\) such that \[ \{g(x_i),g(x_{i+1})\}\subset [x_i,x_{i+1}]. \] 
		\end{enumerate} 
		
		Similarly, for two intervals \(I,J\subset \T\), we say that the orbits of \(I\) and \(J\) are \emph{unlinked with respect to \(G\)} if, for every \(x\in I\) and \(y\in J\), the (partial) orbits of \(x\) and \(y\) are unlinked with respect to \(G\). 
		
		A cyclically ordered collection of intervals \((I_1,\ldots,I_n)\) is said to be \emph{properly unlinked with respect to \(G\)} if \((x_1,\ldots,x_n)\) is properly unlinked for every choice of distinct points \(x_i\in I_i\).
	\end{definition} 
	Observe that a single point \((x_1)\) is always properly unlinked.
	
	\begin{lemma}\label{l.properly_unlinked}
		Let \((x_1,\ldots,x_n)\) be a cyclically ordered \(n\)-tuple of points which is properly
		unlinked with respect to a subset \(G\subset \homeo_+(\T)\). Then, for every
		\(g\in G\) such that \(\fix(g)\cap\{x_1,\ldots,x_n\}=\emptyset\), there exists a pair of
		consecutive points \(\{x_i,x_{i+1}\}\) such that
		\[
		\{g(x_1),\ldots,g(x_n)\}\subset [x_i,x_{i+1}].
		\]
	\end{lemma}
	
	\begin{proof}
		Fix \(g\in G\) such that \(\fix(g)\cap\{x_1,\ldots,x_n\}=\emptyset\). By the definition of
		properly unlinked, there exists \(k\in\{1,\ldots,n\}\) such that
		\[
		\{g(x_k),g(x_{k+1})\}\subset [x_k,x_{k+1}].
		\]
		Let \(t\in\{1,\ldots,n\}\setminus\{k,k+1\}\). Since the pairs \(\{x_t,x_k\}\) and
		\(\{x_t,x_{k+1}\}\) are unlinked with respect to \(G\), and since
		\[
		g(x_k)\in [x_k,x_{k+1}]\subset [x_k,x_t],
		\qquad
		g(x_{k+1})\in [x_k,x_{k+1}]\subset [x_t,x_{k+1}],
		\]
		it follows that \(g(x_t)\) belongs to both intervals \([x_k,x_t]\) and
		\([x_t,x_{k+1}]\). Therefore
		\[
		g(x_t)\in [x_k,x_t]\cap [x_t,x_{k+1}] = [x_k,x_{k+1}].
		\]
		Since this holds for every \(t\notin\{k,k+1\}\), we conclude that
		\[
		\{g(x_1),\ldots,g(x_n)\}\subset [x_k,x_{k+1}],
		\]
		as required.
	\end{proof}

	\begin{lemma}\label{l.mobius_ping-pong}
		Let \(F,G\leq \homeo_+(\T)\), set \(H:=\langle F,G\rangle\) and \(S:=F\cap G\), and let
		\((U,V)\) be a proper ping-pong partition for \(F\) and \(G\). Assume that the action of
		\(H\) on \(\T\) is minimal. 
		
		If the collection of intervals \(U\) is properly unlinked with
		respect to \(F\smallsetminus S\), then every element \(h\in H\) which is not conjugate
		in \(H\) to an element of \(F\cup G\) is M\"obius-like.
	\end{lemma}
	
	\begin{proof}
		Set \(F^*=F\smallsetminus S\) and \(G^*=G\smallsetminus S\). 
		After the coarsening procedure used in the proof of Lemma~\ref{l.fixed_ping-pong}, we may
		assume that the connected components of \(U\) and \(V\) alternate around the circle.
		Since the action is minimal, Lemma~\ref{l.minimal_ping-pong} implies that
		\[
		\partial U=\partial V=\{x_k\}_{k=1}^{2n},
		\]
		where the points \(x_k\) are cyclically ordered and we have the following connected components
		\[
		I_k:=(x_{2k-1},x_{2k})\subset U,
		\qquad
		J_k:=(x_{2k},x_{2k+1})\subset V,
		\qquad
		k=1,\ldots,n.
		\] 
		Moreover, the properly unlinked hypothesis is preserved under the previous reduction.
		Since \(U\) is properly unlinked with respect to \(F^*\), any choice of points in distinct
		connected components of \(U\) is properly unlinked. In particular, the cyclically ordered
		\(n\)-tuples of boundary points from \(\overline U\),
		\[
		(x_1,x_3,\ldots,x_{2n-1})
		\qquad\text{and}\qquad
		(x_2,x_4,\ldots,x_{2n}),
		\]
		are properly unlinked with respect to \(F^*\).

		We claim that, for every \(f\in F^*\), the image \(f(U)\) is contained in a single connected component of \(V\). Indeed, when \(n=1\) the claim is trivial. Assume therefore that \(n\ge2\), and fix \(f\in F^*\). Since 
		\[ \fix(f)\cap\{x_1,\ldots,x_{2n}\}=\emptyset, \] 
		and since the cyclically ordered \(n\)-tuple \((x_1,x_3,\ldots,x_{2n-1})\) is properly unlinked with respect to \(F^*\), Lemma~\ref{l.properly_unlinked} yields an index \(p\in\{1,\ldots,n\}\) such that 
		\[ f(x_{2j-1})\in [x_{2p-1},x_{2p+1}] \qquad\text{for every }j\in\{1,\ldots,n\}. \] 
		Now each point \(x_{2j-1}\) is an endpoint of a connected component of \(U\), hence 
		\[ f(x_{2j-1})\in [x_{2p-1},x_{2p+1}]\cap \overline{V} =[x_{2p},x_{2p+1}] \qquad\text{for every }j\in\{1,\ldots,n\}. \] 
		Similarly, the cyclically ordered \(n\)-tuple \((x_2,x_4,\ldots,x_{2n})\) is properly unlinked with respect to \(F^*\). Applying Lemma~\ref{l.properly_unlinked} once more, we find an index \(q\in\{1,\ldots,n\}\) such that 
		\[ f(x_{2j})\in [x_{2q},x_{2q+2}]\cap \overline{V} =[x_{2q},x_{2q+1}] \qquad\text{for every }j\in\{1,\ldots,n\}. \] 
		Since \(f\) is orientation-preserving, the relative circular order of the points \(\{x_1,\ldots,x_{2n}\}\) is preserved. Hence one of the odd points \(f(x_{2j-1})\) must also lie in \([x_{2q},x_{2q+1}]\), which forces 
		\[ [x_{2p},x_{2p+1}]=[x_{2q},x_{2q+1}]. \] 
		Therefore 
		\[ f(x_i)\in [x_{2q},x_{2q+1}] \qquad\text{for every }i\in\{1,\ldots,2n\}, \] 
		and consequently 
		\[ f(U)\subset V\cap [x_{2q},x_{2q+1}]=J_{q}. \] 
		This proves the claim. We now complete the proof of lemma~\ref{l.mobius_ping-pong}. As in lemma~\ref{l.fixed_ping-pong}, let
		\[
		\tilde h=g_k f_k\cdots g_1 f_1,
		\qquad
		f_i\in F^*,\ g_i\in G^*,
		\]
		be a conjugate of \(h\) written as a cyclically reduced alternating word. By the claim,
		\(f_1(U)\) is contained in a single connected component of \(V\), and therefore
		\(\tilde h(U)\) is contained in a single connected component \(I\) of \(U\). Hence every
		other connected component of \(U\) contains no fixed point of \(\tilde h\). By the argument
		above, we obtain
		\[
		\#\bigl(\fix(\tilde h)\cap \overline U\bigr)\le 1.
		\]
		
		To obtain the corresponding bound on \(\overline V\), we apply the same reasoning to
		\(\tilde h^{-1}\). Since the reduced expression of \(\tilde h^{-1}\) ends with \(f_1^{-1}\),
		the claim gives that \(f_1^{-1}(U)\) is contained in a single connected component of \(V\).
		Hence \(\tilde h^{-1}(V)\) is contained in a single connected component of \(V\), and
		therefore
		\[
		\#\bigl(\fix(\tilde h)\cap \overline V\bigr)\le 1.
		\]
		
		Since \(\T=\overline{U\cup V}\), it follows that
		\[
		\#\fix(h)=\#\fix(\tilde h)
		=\#\bigl(\fix(\tilde h)\cap (\overline U\cup \overline V)\bigr)\le 2.
		\]
		
		Moreover, \(\tilde h\) always has at least one fixed point. Indeed, by the claim, \(\tilde h(U)\) is contained in a single connected component \(I\) of \(U\), hence \(\tilde h(I)\subset I\). Therefore the nested family 
		\[ \tilde h^n(\overline I)\subset \overline I \] 
		has nonempty intersection, which consists of a fixed point of \(\tilde h\). 
		
		Now let \(x\in\fix(\tilde h)\). If \(x\in U^\circ\), then \(x\) is attracting; if \(x\in V^\circ\), then \(x\) is repelling. Thus any parabolic fixed point of \(\tilde h\) must lie in \(\partial U=\partial V\). In particular, a parabolic fixed point belongs to both \(\overline U\) and \(\overline V\). Since 
		\[ \#\bigl(\fix(\tilde h)\cap \overline U\bigr)\le 1 \qquad\text{and}\qquad \#\bigl(\fix(\tilde h)\cap \overline V\bigr)\le 1, \] 
		there can be at most one parabolic fixed point. Moreover, a parabolic fixed point cannot coexist with any other fixed point, for otherwise either \(\overline U\) or \(\overline V\) would contain at least two fixed points. Therefore \(\tilde h\) has either a unique fixed point, which is parabolic, or exactly two fixed points, one attracting and one repelling. In both cases \(\tilde h\) is M\"obius-like.
	\end{proof}
	
	\begin{theorem}\label{t.mobius_like_trivial_stabilizer}
		Let \(F,G\leq \homeo_+(\T)\) be countable M\"obius-like subgroups containing no element
		with irrational rotation number, and let
		\[
		\bar x=(x_1,\ldots,x_n),\qquad \bar y=(y_1,\ldots,y_n)
		\]
		be cyclically ordered \(n\)-tuples of points in \(\T\) such that, for all distinct
		\(i,j\in\{1,\ldots,n\}\),
		\begin{enumerate}[\rm i.]
			\item \(x_i\notin F.x_j\) and \(y_i\notin G.y_j\);
			\item \(\stab_F(x_i)=\stab_G(y_i)=\{\mathrm{id}\}\);
			\item \((x_1,\ldots,x_n)\) is properly unlinked with respect to
			\(F\smallsetminus\{\mathrm{id}\}\);
			\item \(|F|>2\) or \(|G|>2\).
		\end{enumerate}
		Let \(\sigma\) be any cyclic order-preserving permutation of \(\{1,\ldots,n\}\). Then the
		minimal amalgamated product
		\[
		H=(F,\bar x)\star_{\sigma}(G,\bar y)
		\]
		is M\"obius-like.
	\end{theorem}
	
	\begin{proof}
		Since the stabilizers are trivial, the amalgamated product is well defined, and by
		Theorem~\ref{t.action_product} it yields a minimal action \(H=(F,\bar x)\star_{\sigma}(G,\bar y)\) on \(\T\).
		
		Let \((U,V)\) be the proper ping-pong partition associated to this action. By
		Lemma~\ref{l.mobius_ping-pong}, every element of \(H\) which is not conjugate into
		\(\Psi(F)\cup\Psi(G)\) is M\"obius-like.
		
		It remains to consider the elements which are conjugate into \(\Psi(F)\cup\Psi(G)\).
		Since \(F\) and \(G\) are M\"obius-like, it is enough to understand how the blow-up affects
		the dynamics of elements of \(F\) and \(G\). Observe that no nontrivial element fixes any of
		the blown-up points \(x_i\) or \(y_i\). Therefore, if such an element has fixed points, then the blow-up does not change either the number of fixed points or their dynamical type: hyperbolic fixed points remain hyperbolic, and parabolic fixed points remain parabolic. Hence every element of
		\(\Psi(F)\cup\Psi(G)\) with fixed points is still M\"obius-like.
		
		Thus the only possible obstruction comes from fixed-point-free elements of
		\(\Psi(F)\cup\Psi(G)\), that is, from blow-ups of rotations. Let \(\gamma\in \Psi(F)\cup\Psi(G)\) be fixed-point-free, and let \(r\in F\cup G\) be its
		projection. Since \(r\) is fixed-point-free and M\"obius-like, it is topologically conjugate to a
		rotation. By hypothesis, \(r\) has no irrational rotation number, so its rotation number is
		rational, say
		\[
		\rho(r)=\frac pq,\qquad \gcd(p,q)=1.
		\]
		Therefore \(r\) is topologically conjugate to the rigid rotation \(R_{p/q}\), and in particular \(r^q=\mathrm{id}\).
		
		Now \(\gamma\) is a blow-up of \(r\). Since \(r^q=\mathrm{id}\), the element \(\gamma^q\)
		projects to the identity. If \(\gamma^q\neq \mathrm{id}\), then \(\gamma^q\) acts
		nontrivially on some blown-up interval, which implies that the corresponding point of
		\(\T\) has nontrivial stabilizer for the action of \(F\) or \(G\). This contradicts the
		assumption that all stabilizers are trivial. Therefore
		\[
		\gamma^q=\mathrm{id}.
		\]
		Hence \(\gamma\) is an orientation-preserving finite-order homeomorphism of the circle, and
		therefore it is topologically conjugate to a rational rotation. In particular, \(\gamma\)
		is M\"obius-like.
		
		We conclude that every element of \(H\) is M\"obius-like. Therefore the action of \((F,\bar x)\star_{\sigma}(G,\bar y)\) is M\"obius-like.
	\end{proof}
	
	\begin{remark}
		The converse also holds: if the minimal amalgamated product \((F,\bar x)\star_{\sigma}(G,\bar y)\) is M\"obius-like, then \(F\) and \(G\) are M\"obius-like and contain no element with
		irrational rotation number.
	\end{remark}
	
	\begin{proof}
		If \(F\) or \(G\) were not M\"obius-like, then the corresponding subgroup
		\(\Psi(F)\) or \(\Psi(G)\) in the amalgamated product would also contain a non
		M\"obius-like element, contradicting the fact that the whole action is M\"obius-like.
		
		Likewise, if \(F\) or \(G\) contained an element with irrational rotation number, then
		its blow-up in \(\Psi(F)\) or \(\Psi(G)\) would still have irrational rotation number.
		Since \(\Psi(F)\) and \(\Psi(G)\) preserve a proper ping-pong partition, such an element
		cannot be topologically conjugate to an irrational rotation; it is a Denjoy-type example,
		hence not M\"obius-like; see, for instance,	\cite{ghys-circle, Kim-Koberda-Denjoy}. This again contradicts the hypothesis.
		
		Therefore both \(F\) and \(G\) are M\"obius-like and contain no element with irrational
		rotation number.
	\end{proof}
	
	We now turn to the case of singleton blow-ups with nontrivial stabilizers. In this setting,
	the M\"obius-like character of the amalgamated product depends on the way the stabilizers
	are identified. We begin with a lemma describing when the stabilizer subgroup produces
	bi-parabolic elements, and then derive the corresponding characterization theorem.
	
	\begin{lemma}\label{l.singleton_stabilizer_biparabolic}
		With notation as in Subsection~\ref{d.action_product}, let \(H=(F,\{x\})\star_{\theta}(G,\{y\})\) be an amalgamated product of circle actions, and assume that
		\(S_F=\stab_F(x)\) and \(S_G=\stab_G(y)\) are nontrivial.
		Then an element \(\gamma_s\in \Psi(S_F)=\Psi(S_G)\) is bi-parabolic if and only if
		\(\theta\) preserves the order of \(s\in S_F\).
		
		In particular, \(\Psi(S_F)\) contains a bi-parabolic element if and only if
		\(\theta\) is not order-inverting.
	\end{lemma}
	
	\begin{proof}
		The associated proper ping-pong partition has the form
		\[
		\bigl(h_F^{-1}(x),\,h_G^{-1}(y)\bigr).
		\]
		Moreover, every nontrivial element of \(S_F\) fixes \(x\) and no other point, and
		similarly every nontrivial element of \(S_G\) fixes \(y\) and no other point. Hence the
		actions of \(S_F\) on \(\T\smallsetminus\{x\}\) and of \(S_G\) on
		\(\T\smallsetminus\{y\}\) are fixed-point-free.
		
		Let \(s\in S_F\), and denote by \(\gamma_s\in \Psi(S_F)=\Psi(S_G)\) the corresponding
		element in the amalgamated product. Observe that the sign of \(\gamma_s(z)-z\) is constant on \(\core(h_F)\) and extends to the whole interval \(h_G^{-1}(y)\); similarly, the sign of \(\gamma_s(z)-z\) is constant on \(\core(h_G)\) and extends to the whole interval \(h_F^{-1}(x)\).
		
		On \(\core(h_F)\), the action of \(\gamma_s\) is given by
		\[
		\gamma_s=h_F^{-1}\,s\,h_F,
		\]
		while on \(\core(h_G)\) it is given by
		\[
		\gamma_s=h_G^{-1}\,\theta(s)\,h_G.
		\]
		Therefore, the signs of \(\gamma_s(z)-z\) on \(h_G^{-1}(y)\) and \(h_F^{-1}(x)\) are
		opposite if and only if \(\theta\) is order-inverting.
		
		It follows that \(\gamma_s\) fixes exactly the two boundary points of the partition, and
		both fixed points are parabolic if and only if the signs on the two sides agree, that is,
		if and only if \(\theta\) preserves the order of \(s\).
		
		Thus \(\gamma_s\) is bi-parabolic if and only if \(\theta\) preserves the order of \(s\).
		In particular, \(\Psi(S_F)=\Psi(S_G)\) contains a bi-parabolic element if and only if
		\(\theta\) is not order-inverting.
	\end{proof}
	
	\begin{theorem}\label{t.mobius_like_singleton_stabilizer}
		Let \(F,G\leq \homeo_+(\T)\) be countable M\"obius-like subgroups containing no element
		with irrational rotation number, and let \(x,y\in\T\). Set
		\[
		S_F:=\stab_F(x),\qquad S_G:=\stab_G(y).
		\]
		Assume that:
		\begin{enumerate}[\rm i.]
			\item \(\fix(s_f)=\{x\}\) for every \(s_f\in S_F\), and \(\fix(s_g)=\{y\}\) for every
			\(s_g\in S_G\);
			\item \(S_F\lneqq F\), \(S_G\lneqq G\), and \([F:S_F]>2\);
			\item there exists an isomorphism \(\theta:S_F\xrightarrow{\sim} S_G\).
		\end{enumerate}
		Then the minimal amalgamated product
		\[
		H=(F,x)\star_{\theta}(G,y)
		\]
		is M\"obius-like if and only if \(\theta\) is order-inverting.
	\end{theorem}
	
	\begin{proof}
		If \(S_F=S_G=\{\mathrm{id}\}\), then \(\theta\) is vacuously order-inverting, and the
		conclusion follows from Theorem~\ref{t.mobius_like_trivial_stabilizer}. We may therefore
		assume that the stabilizers are nontrivial.
		
		Let \(H=(F,x)\star_{\theta}(G,y)\), with \(\Psi\) as in Definition~\ref{d.action_product}. Since the singleton \((x)\) is automatically properly unlinked, Lemma~\ref{l.mobius_ping-pong} implies that every element of \(H\) which is not conjugate into \(\Psi(F)\cup\Psi(G)\) is M\"obius-like.
		
		For elements of \(\Psi(F)\cup\Psi(G)\) outside \(\Psi(S_F)=\Psi(S_G)\), the same argument as in
		the proof of Theorem~\ref{t.mobius_like_trivial_stabilizer} shows that they are M\"obius-like.
		
		Thus it only remains to consider the subgroup \(\Psi(S_F)=\Psi(S_G)\). These elements fix exactly the two endpoints of the proper ping-pong partition and, by Lemma~\ref{l.singleton_stabilizer_biparabolic}, this subgroup contains a bi-parabolic element if and only if \(\theta\) is not order-inverting. Since bi-parabolic elements are the only orientation-preserving homeomorphism with 2 fixed points which are not topologically conjugate to \(\psl(2,\RR)\), it follows that \(H\) is M\"obius-like if and only if \(\theta\) is order-inverting.
	\end{proof}
	
		\begin{remark}
			Theorem~\ref{t.mobius_like_singleton_stabilizer} raises the problem of deciding for which
			stabilizers \(S_F\) and \(S_G\) there exists an order-inverting isomorphism
			\(\theta:S_F\xrightarrow{\sim} S_G\). As seen in the proof of
			Lemma~\ref{l.singleton_stabilizer_biparabolic}, this forces \(S_F\) and \(S_G\) to act on
			\(\RR\) without fixed points. Hence, by Hölder's theorem, every subgroup of \(\homeo_+(\RR)\) acting freely on \(\RR\) is abelian	and semi-conjugate to a subgroup of \(\mathrm{Isom}_+(\RR)\); see, for instance, \cite[Section~2]{ghys-circle} or \cite[Chapter~2]{navas-book}. See also \cite{carnevale2022groups} for a related extension.
			
			Thus, if \(S_F\cong A_F\subset \RR\) and \(S_G\cong A_G\subset \RR\), then every
			isomorphism \(\theta:S_F\to S_G\) induces an isomorphism \(\alpha_\theta:A_F\to A_G\), and
			\(\theta\) is order-inverting if and only if \(\alpha_\theta\) is. Since there is a natural
			bijection between order-preserving and order-inverting isomorphisms, the existence of an
			order-inverting \(\theta\) is equivalent to the existence of an order-preserving one.
			Accordingly, finding stabilizers for which the amalgamated product can be M\"obius-like
			amounts to finding two abelian subgroups of \(\RR\) which are isomorphic by an
			order-preserving isomorphism.
		\end{remark}

	\subsection{Conditions for being non-conjugate into $\psl^{(k)}(2,\RR)$}
	
	We now discuss sufficient conditions ensuring that an amalgamated product of circle
	actions is not topologically conjugate to a subgroup of \(\psl^{(k)}(2,\RR)\). By the
	convergence group theorem, this amounts to showing that the resulting action fails the
	convergence property. We will treat two mechanisms separately. The first is based on the
	presence of a nondiscrete factor, and yields a broad obstruction. The second is adapted to
	the finitely generated setting, where nondiscreteness is too restrictive, and is based
	instead on conical limit points in finite lifts of M\"obius actions.
	
	For convenience, we recall the following theorem.
	
	\begin{theorem*}[Tukia--Gabai--Casson--Jungreis]
		A subgroup \(G\leq \homeo_+(\T)\) is topologically conjugate to a subgroup of
		\(\psl(2,\RR)\) if and only if \(G\) is a convergence group. Equivalently, every
		sequence \(\{g_n\}\subset G\) admits a subsequence \(\{g_{n_k}\}\) such that one of the
		following alternatives holds:
		\begin{enumerate}[\rm a.]
			\item\label{a.convergence_group} there exist points \(x,y\in\T\) such that
			\[
			g_{n_k}\to y \quad\text{pointwise on }\T\smallsetminus\{x\},
			\qquad
			g_{n_k}^{-1}\to x \quad\text{pointwise on }\T\smallsetminus\{y\};
			\]
			\item\label{b.convergence_group} there exists \(g\in \homeo_+(\T)\) such that
			\[
			g_{n_k}\to g \quad\text{pointwise on }\T,
			\qquad
			g_{n_k}^{-1}\to g^{-1} \quad\text{pointwise on }\T.
			\]
		\end{enumerate}
	\end{theorem*}
	
	The convergence group theorem is due to Tukia, Gabai, and Casson--Jungreis; see
	\cite{Tukia,Gabai,Casson-Jungreis}.
	
	We now revisit the setting of Theorem~\ref{t.main_examples_2n}, but impose the additional
	assumption that one of the factors acts nondiscretely on \(\T\). This allows us to produce
	sequences in the amalgamated product whose limit behavior is incompatible with the
	convergence property, and hence with topological conjugacy into \(\psl^{(k)}(2,\RR)\).
	
	\begin{theorem}\label{t.main_examples_non_psl}
		Let \(F,G\leq \homeo_+(\T)\) be countable subgroups, and let
		\[
		\bar x=(x_1,\ldots,x_n),\qquad \bar y=(y_1,\ldots,y_n)
		\]
		be cyclically ordered \(n\)-tuples of points in \(\T\). Assume that, for all distinct
		\(i,j\in\{1,\ldots,n\}\),
		\begin{enumerate}[\rm i.]
			\item \(x_j\notin F.x_i\) and \(y_j\notin G.y_i\);
			\item \(\stab_F(x_i)=S_F\) and \(\stab_G(y_i)=S_G\), where
			\(\theta:S_F\xrightarrow{\sim}S_G\) is a fixed isomorphism;
			\item\(\fix(s_f)=\{x_1,\ldots,x_n\}\) and \(\fix(s_g)=\{y_1,\ldots,y_n\}\) for every nontrivial \(s_f\in S_F\) and \(s_g\in S_G\);
			\item \(S_F\lneqq F\), \(S_G\lneqq G\), and at least one of the indices
			\([F:S_F]\) and \([G:S_G]\) is greater than \(2\).
		\end{enumerate}
		Let \(\sigma\) be any cyclic order-preserving permutation of \(\{1,\ldots,n\}\), and
		assume that one of the actions of \(F\) or \(G\) on \(\T\) is nondiscrete. Then every
		minimal amalgamated product
		\[
		H=(F,\bar x)\star_{\theta,\sigma}(G,\bar y)
		\]
		is not topologically conjugate to a subgroup of \(\psl^{(k)}(2,\RR)\), for any
		\(k\geq 1\).
	\end{theorem}
	
	\begin{proof}
		Assume, without loss of generality, that the action of \(F\) on \(\T\) is
		nondiscrete. Then there exists a sequence of distinct elements
		\(\{f_i\}_{i\in\NN}\subset F\) such that \(f_i\longrightarrow \idid\) pointwise on \(\T\).
		We first show that, after passing to a subsequence if necessary, we may assume that \(f_i\notin S_F\)
		for every \(i\in\NN\).
		
		Indeed, suppose that this is not the case. Then there exists a sequence of distinct
		elements \(\{s_i\}_{i\in\NN}\subset S_F\) such that \(s_i\longrightarrow \idid\) pointwise on \(\T\).
		
		Choose \(f\in F\smallsetminus S_F\), which is possible since \(S_F\lneqq F\). Since the
		points \(x_1,\ldots,x_n\) lie in distinct \(F\)-orbits, we have \(f(x_1)\neq x_j\) for every \(j\in\{1,\ldots,n\}\).
		
		Consider now the conjugate sequence \(f s_i f^{-1}\in F\). It still converges pointwise to the identity on \(\T\), but \(f s_i f^{-1}(x_1)\neq x_1\), so \(f s_i f^{-1}\notin S_F\) for every \(i\). This proves the claim.
		
		We may therefore fix a sequence of distinct elements \(\{f_i\}_{i\in\NN}\subset F\smallsetminus S_F\)
		converging pointwise to the identity on \(\T\). With notation as in	Definition~\ref{d.action_product}, let \(\Psi:F*_\theta G\xrightarrow{\sim} H\) be the associated isomorphism, and let \(h_F\) be semi-conjugacy between the action of \(\Psi(F)\) and the original action of \(F\).
		By the semi-conjugacy, the sequence \(\{\Psi(f_i)\}\) converges pointwise to the identity on \(\core(h_F)\).
		
		Now consider the blown-up interval \(I_{x_1}:=h_F^{-1}(x_1)\). Since \(f_i\notin S_F\),
		we have \(f_i(x_1)\neq x_1\) for every \(i\), while \(f_i(x_1)\to x_1\). After passing to a further subsequence, we may assume that the points \(f_i(x_1)\) are pairwise distinct and converge monotonically to \(x_1\) from one side. For each \(i\), the map \(\Psi(f_i)\) sends \(I_{x_1}\) onto the interval
		\[
		I_{f_i(x_1)}:=h_F^{-1}(f_i(x_1)).
		\]
		Since the fibers \(I_{f_i(x_1)}\) are pairwise disjoint whenever the points \(f_i(x_1)\) are
		distinct, and they all lie in the compact circle, their diameters tend to \(0\). Hence
		\(\Psi(f_i)|_{I_{x_1}}\) converges to a constant map, namely to the endpoint of
		\(I_{x_1}\) determined by the chosen monotone approach of \(f_i(x_1)\) to \(x_1\).
		
		Therefore the sequence \(\{\Psi(f_i)\}_{i\in\NN}\subset H\) converges pointwise to the
		identity on \(\core(h_F)\), but on the interval \(I_{x_1}\) it converges to a constant map.
		In particular, \(\{\Psi(f_i)\}_{i\in\NN}\) does not converge to a homeomorphism of \(\T\).
		
		We first prove that \(H\) is not topologically conjugate to a subgroup of
		\(\psl(2,\RR)\). 
		
		Since every fiber of \(h_F\) meets \(\core(h_F)\), the restriction Since every fiber of \(h_F\) meets \(\core(h_F)\), the restriction \(h_F|_{\core(h_F)}:\core(h_F)\to \T\) is surjective. In particular, \(\core(h_F)\) is uncountable, so we may choose as many	distinct points in \(\core(h_F)\) as needed. Choose three distinct points \(z_1,z_2, z_3\in \core(h_F)\). Since \(\Psi(f_i)\to\idid\) on \(\core(h_F)\), the sequences \(\Psi(f_i)(z_1)\), \(\Psi(f_i)(z_2)\) and \(\Psi(f_i)(z_3)\) converge to three distinct points of \(\T\). Hence the sequence \(\{\Psi(f_i)\}\) does not satisfy alternative~\ref{a.convergence_group} of the convergence group theorem, and, as already observed, it does not satisfy alternative~\ref{b.convergence_group} either. Therefore \(H\) is not a convergence group, and by the convergence group theorem it is not
		topologically conjugate to a subgroup of \(\psl(2,\RR)\).
		
		Now fix \(k\geq 2\), and suppose by contradiction that \(H\) is topologically
		conjugate to a subgroup of \(\psl^{(k)}(2,\RR)\). Then there exists a periodic
		homeomorphism \(\tau\) of order \(k\) commuting with \(H\), such that the induced
		action of \(H\) on the quotient circle \(\T/\langle\tau\rangle\) is topologically
		conjugate to the action of a subgroup of \(\psl(2,\RR)\). Choose
		\(M>2k\) distinct points \(z_1,\ldots,z_M\in \core(h_F)\). Their images under \(\Psi(f_i)\)
		converge to \(M\) distinct points of \(\T\), and hence their projections to the
		quotient circle converge to at least \(M/k>2\) distinct points. As before, this is
		incompatible with the convergence property. Thus the induced action on
		\(\T/\langle\tau\rangle\) cannot be conjugate to a subgroup of \(\psl(2,\RR)\), a
		contradiction.
		
		We conclude that \(H\) is not topologically conjugate to a subgroup of
		\(\psl^{(k)}(2,\RR)\) for any \(k\geq 1\).
	\end{proof}

	The previous theorem yields very few finitely generated M\"obius-like examples. Indeed, if
	the nondiscrete factor, say \(F\), is finitely generated and non-elementary, then by
	\cite[Theorem~C and Remark~1.7]{BonattiCarnevaleTriestino2024} it contains an element
	topologically conjugate to an irrational rotation. By the results of the previous
	subsection, the blow-up of such an element yields a Denjoy-type homeomorphism, so the
	corresponding amalgamated product cannot be M\"obius-like. Thus, in the finitely generated
	M\"obius-like setting, a nondiscrete factor in
	Theorem~\ref{t.main_examples_non_psl} must be elementary.
	
	This motivates a second approach, in which the failure of the convergence property is
	detected not by nondiscreteness, but by a distinguished accumulation behavior at a
	blown-up point. We begin with the relevant notion.

	\begin{definition}
		Let \(F\leq \homeo_+(\T)\), and let \(x\in\T\). We say that \(x\) is a \emph{conical limit point} for the action of \(F\) (also called a \emph{point of approximation} in older terminology) if there exists a sequence of distinct elements \(\{f_n\}\subset F\) and two distinct points \(a,b\in\T\) such that
		\[
		f_n(x)\to a
		\qquad\text{and}\qquad
		f_n|_{\T\smallsetminus\{x\}}\to b
		\quad\text{pointwise.}
		\]
	\end{definition}
	
	Conical limit points are abundant in many classical settings. For instance, if \(F\) is
	cocompact Fuchsian, then every point of \(\T\) is conical. More generally, if \(F\) is a
	finitely generated Fuchsian group, then every point of its limit set \(\Lambda(F)\) which
	is not a parabolic fixed point is conical. In particular, if \(F\) is of the first kind,
	then all points of \(\T\) except the countable set of parabolic fixed points are conical.
	
	To exploit conical limit points in the amalgamated product, we first prove a descent principle for actions topologically conjugate to subgroups of \(\psl^{(k)}(2,\RR)\).
	
	\begin{lemma}\label{l.k_lift_descends} With notation as in Definition~\ref{d.action_product}, let \( H=(F,\bar x)\star_{\theta,\sigma}(G,\bar y) \) be the minimal amalgamated product of \(F\) and \(G\). If \(H\) is topologically conjugate to a subgroup of \(\psl^{(k)}(2,\RR)\), then \(F\) and \(G\) are also topologically conjugate to subgroups of \(\psl^{(k)}(2,\RR)\). 
	\end{lemma} 
	
	\begin{proof} 
		We first treat the case \(k=1\). Assume that \(H\) is topologically conjugate to a subgroup of \(\psl(2,\RR)\). Then \(H\) is a convergence group, and therefore so is the subgroup \(\Psi(F)\leq H\).
		
		Let \(\{f_n\}\subset F\) be any sequence of distinct elements. Since \(\Psi:F*_\theta G\xrightarrow{\sim} H\) is injective, the sequence \(\{\Psi(f_n)\}\subset \Psi(F)\) also consists of distinct elements. Passing to a subsequence, one of the two alternatives in the convergence group theorem holds for \(\{\Psi(f_n)\}\). We show that the same alternative then holds for \(\{f_n\}\). This will prove that \(F\) is a convergence group and therefore topologically conjugate to a subgroup of \(\psl(2,\RR)\). 
		
		\smallskip 
		\noindent\textbf{Case a.} Assume that there exist points \(x^\ast,y^\ast\in\T\) such that \[ \Psi(f_n)\to y^\ast \quad\text{pointwise on }\T\smallsetminus\{x^\ast\}, \qquad \Psi(f_n)^{-1}\to x^\ast \quad\text{pointwise on }\T\smallsetminus\{y^\ast\}. \] Let \(h_F:\T\to\T\) be the semi-conjugacy between the action of \(\Psi(F)\) and the original action of \(F\), and set \[ x:=h_F(x^\ast),\qquad y:=h_F(y^\ast). \] We claim that \[ f_n\to y \quad\text{pointwise on }\T\smallsetminus\{x\}, \qquad f_n^{-1}\to x \quad\text{pointwise on }\T\smallsetminus\{y\}. \] Fix \(u\in\T\smallsetminus\{x\}\). Since \(h_F(x^\ast)=x\), the fiber \(h_F^{-1}(u)\) does not contain \(x^\ast\). Choose any \(z\in h_F^{-1}(u)\). Then \(z\neq x^\ast\), so \[ \Psi(f_n)(z)\to y^\ast. \] Applying \(h_F\) and using the semi-conjugacy relation \(h_F\circ\Psi(f_n)=f_n\circ h_F\), we obtain \[ f_n(u)=h_F(\Psi(f_n)(z))\longrightarrow h_F(y^\ast)=y. \] The proof for the inverses is identical. Thus \(\{f_n\}\) satisfies alternative~a. 
		
		\smallskip 
		\noindent\textbf{Case b.} Assume that there exists \(g\in\homeo_+(\T)\) such that \[ \Psi(f_n)\to g \quad\text{pointwise on }\T, \qquad \Psi(f_n)^{-1}\to g^{-1} \quad\text{pointwise on }\T. \] We first show that \(g\) preserves the fibers of \(h_F\). Indeed, if \(z_1,z_2\in\T\) satisfy \(h_F(z_1)=h_F(z_2)\), then for every \(n\), \[ h_F(\Psi(f_n)(z_1)) = f_n(h_F(z_1)) = f_n(h_F(z_2)) = h_F(\Psi(f_n)(z_2)). \] Passing to the limit gives \[ h_F(g(z_1))=h_F(g(z_2)). \] Hence the map \[ \bar g:\T\to\T,\qquad \bar g(u):=h_F(g(z)) \quad\text{for }z\in h_F^{-1}(u), \] is well defined. Similarly, \(g^{-1}\) also preserves the fibers of \(h_F\), so \(\bar g\in\homeo_+(\T)\). Now fix \(u\in\T\), and choose any \(z\in h_F^{-1}(u)\). Then \[ f_n(u)=h_F(\Psi(f_n)(z))\longrightarrow h_F(g(z))=\bar g(u). \] So \(f_n\to \bar g\) pointwise on \(\T\). Applying the same argument to the inverses, we also get \[ f_n^{-1}\to \bar g^{-1} \quad\text{pointwise on }\T. \] Thus \(\{f_n\}\) satisfies alternative~b. 
		
		In either case, every sequence of distinct elements of \(F\) admits a subsequence satisfying one of the convergence-group alternatives. Hence \(F\) is a convergence group. The same argument applies to \(G\). 
		
		\smallskip 
		We now treat the case \(k\geq 2\). Assume that \(H\) is topologically conjugate to a
		subgroup of \(\psl^{(k)}(2,\RR)\). Then there exists a periodic homeomorphism \(\tau\) of
		order \(k\), commuting with \(H\), such that the induced action of \(H\) on the quotient
		circle \(\T/\langle\tau\rangle\) is a convergence group. We claim that \(\tau\) preserves
		the fibers of \(h_F\) and \(h_G\).
		
		Indeed, since we work with the minimal amalgamated product, every wandering interval which
		does not arise from the original blow-up construction has already been collapsed. Thus the
		only surviving nontrivial intervals canonically distinguished by the action are precisely
		the fibers of \(h_F\) and \(h_G\). Because \(\tau\) commutes with the whole group \(H\), it preserves the action-product structure, and therefore permutes these distinguished intervals. Hence \(\tau\) preserves the fibers of \(h_F\)
		and of \(h_G\).
		
		It follows that there exist periodic homeomorphisms \(\bar\tau_F\) and \(\bar\tau_G\) of order \(k\) such that 
		\[ h_F\circ\tau=\bar\tau_F\circ h_F, \qquad h_G\circ\tau=\bar\tau_G\circ h_G. \] 
		In particular, \(h_F\) and \(h_G\) descend to continuous monotone degree-one maps \[ \widehat h_F:\T/\langle\tau\rangle\longrightarrow \T/\langle\bar\tau_F\rangle, \qquad \widehat h_G:\T/\langle\tau\rangle\longrightarrow \T/\langle\bar\tau_G\rangle, \] which semi-conjugate the quotient action of \(\Psi(F)\) to the quotient action of \(F\), and the quotient action of \(\Psi(G)\) to the quotient action of \(G\). 
		
		Now the quotient action of \(H\) on \(\T/\langle\tau\rangle\) is a convergence group by hypothesis. Therefore its subgroups, in particular the quotient actions of \(\Psi(F)\) and \(\Psi(G)\), are also convergence groups. Applying the case \(k=1\) to these quotient actions, we conclude that the quotient actions of \(F\) and \(G\) on \(\T/\langle\bar\tau_F\rangle\) and \(\T/\langle\bar\tau_G\rangle\) are convergence groups. Hence \(F\) and \(G\) are \(k\)-lifted convergence groups, and therefore topologically conjugate to subgroups of \(\psl^{(k)}(2,\RR)\). 
	\end{proof}

	\begin{corollary}\label{c.k_lift_obstruction}
		With notation as in Definition~\ref{d.action_product}, let \(H=(F,\bar x)\star_{\theta,\sigma}(G,\bar y)\) be the minimal amalgamated product of \(F\) and \(G\). If at least one
		of the groups \(F\) or \(G\) is not topologically conjugate to a subgroup of
		\(\psl^{(k)}(2,\RR)\), then \(H\) is not topologically conjugate to a subgroup of
		\(\psl^{(k)}(2,\RR)\).
	\end{corollary}
	
	We now restrict to the case where the factor \(F\) is itself topologically conjugate to a
	subgroup of \(\psl^{(k)}(2,\RR)\). The next lemma shows that the existence of a
	\(k\)-lifted conical limit point already determines this lift order uniquely.

	\begin{lemma}\label{l.k_lifted_conical_hyperbolic} 
		Assume that \(F\le \homeo_+(\T)\) is topologically conjugate to a subgroup of \(\psl^{(k)}(2,\RR)\), and that \(x\in\T\) is a \(k\)-lifted conical limit point for the action of \(F\). Then \(F\) contains a hyperbolic element with exactly \(2k\) fixed points. In particular, if \(F\) is also topologically conjugate to a subgroup of \(\psl^{(j)}(2,\RR)\) for some \(j\ge 1\), then \(j=k\). 
	\end{lemma} 
	
	\begin{proof} 
		By definition of a \(k\)-lifted conical limit point, there exists a periodic homeomorphism \(\tau\) of order \(k\), commuting with \(F\), such that the induced action \(\bar F\) of \(F\) on the quotient circle 
		\[ \pi:\T\longrightarrow \T/\langle\tau\rangle \]
		is topologically conjugate to a subgroup of \(\psl(2,\RR)\), and \(\bar x:=\pi(x)\) is a conical limit point for this quotient action. 
		
		We claim that \(\bar F\) contains a hyperbolic element. Indeed, if every nontrivial element of \(\bar F\) were elliptic or parabolic, then \(\bar F\) would be elementary. But elementary subgroups of \(\psl(2,\RR)\) have no conical limit points, contradicting the fact that \(\bar x\) is conical. Hence \(\bar F\) contains a hyperbolic element \(\bar h\). 
		
		Let \(h\in F\) be any element whose image in the quotient action is \(\bar h\). Since \(F\) is topologically conjugate to a subgroup of \(\psl^{(k)}(2,\RR)\), the element \(h\) is a \(k\)-lift of the hyperbolic element \(\bar h\). Therefore \(h\) is hyperbolic and has exactly \(2k\) fixed points.
		
		For the last assertion, assume that \(F\) is also topologically conjugate to a subgroup of \(\psl^{(j)}(2,\RR)\) for some \(j\ge 1\). Then, under this second conjugacy, the element \(h\) is again sent to a hyperbolic element, since topological conjugacy preserves the dynamical type. But hyperbolic elements of \(\psl^{(j)}(2,\RR)\) have exactly \(2j\) fixed points, whereas \(h\) has exactly \(2k\) fixed points. Hence \(2j=2k\), and therefore \(j=k\).
	\end{proof}
	
	We can now combine the descent lemma with the previous hyperbolicity criterion to obtain
	the desired obstruction for minimal amalgamated products.
	
	\begin{theorem}\label{t.conical_examples_non_psl} 
		Let	\(F,G\leq \homeo_+(\T)\) be countable subgroups and assume that \(F\) is topologically conjugate to a subgroup of
		\(\psl^{(k)}(2,\RR)\). Let
		\[ \bar x=(x_1,\ldots,x_n),\qquad \bar y=(y_1,\ldots,y_n) \]
		be cyclically ordered \(n\)-tuples of points in \(\T\). Assume that, for all distinct \(i,j\in\{1,\ldots,n\}\), 
		\begin{enumerate}[\rm i.] 
			\item \(x_j\notin F.x_i\) and \(y_j\notin G.y_i\); 
			\item \(\stab_F(x_i)=\stab_G(y_i)=\{\idid\}\); 
			\item \(F\) and \(G\) are nontrivial, and at least one of them has order greater than \(2\).
		\end{enumerate} 
		Let \(\sigma\) be any cyclic order-preserving permutation of \(\{1,\ldots,n\}\), and assume, after relabeling if necessary, that \(x_1\) is a \(k\)-lifted conical limit point for the action of \(F\) on \(\T\). Then every minimal amalgamated product 
		\[ H=(F,\bar x)\star_{\sigma}(G,\bar y) \] 
		is not topologically conjugate to a subgroup of \(\psl^{(j)}(2,\RR)\), for any \(j\geq 1\). 
	\end{theorem}

	\begin{proof}
		By Lemma~\ref{l.k_lifted_conical_hyperbolic}, since \(x_1\) is a \(k\)-lifted conical
		limit point for the action of \(F\), the group \(F\) contains a hyperbolic element with
		exactly \(2k\) fixed points. In particular, the same lemma shows that \(F\) can be
		topologically conjugate to a subgroup of \(\psl^{(j)}(2,\RR)\) only when \(j=k\). Thus
		the lift order of \(F\) is uniquely determined.
		
		Assume now by contradiction that the minimal amalgamated product \(H=(F,\bar x)\star_{\sigma}(G,\bar y)\) is topologically conjugate to a subgroup of \(\psl^{(j)}(2,\RR)\) for some \(j\geq 1\).
		Then Lemma~\ref{l.k_lift_descends} implies that \(F\) is also topologically conjugate to
		a subgroup of \(\psl^{(j)}(2,\RR)\). By the previous paragraph, this forces \(j=k\).
		Hence it is enough to rule out conjugacy of \(H\) into \(\psl^{(k)}(2,\RR)\).
		
		Suppose therefore that \(H\) is topologically conjugate to a subgroup of
		\(\psl^{(k)}(2,\RR)\). Let \(\tau\) be the periodic homeomorphism of order \(k\)
		commuting with \(H\), so that the induced action of \(H\) on the quotient circle
		\(\T/\langle\tau\rangle\) is topologically conjugate to a subgroup of \(\psl(2,\RR)\).
		
		By the proof of Lemma~\ref{l.k_lift_descends}, the homeomorphism \(\tau\) preserves the fibers of \(h_F\), and therefore \(h_F\) descends to a semi-conjugacy between the quotient action of \(\Psi(F)\) and the quotient action of \(F\). Since \(x_1\) is a \(k\)-lifted conical limit point for the action of \(F\), its image in the quotient circle is a conical limit point for the quotient action of \(F\). Choose a conical sequence \(\{f_i\}_{i\in\NN}\subset F\) associated to this point. Thus there exist two distinct points \(a,b\) in the quotient circle such that 
		\[ f_i(\bar x_1)\to a, \qquad f_i|_{\T/\langle\tau\rangle\smallsetminus\{\bar x_1\}}\to b \quad\text{pointwise.} \] 
		
		Let \(I_{x_1}:=h_F^{-1}(x_1)\). Since \(\stab_F(x_1)=\{\idid\}\), the points \(f_i(x_1)\) are pairwise distinct. For each \(i\), the map \(\Psi(f_i)\) sends \(I_{x_1}\) onto the interval 
		\[ I_{f_i(x_1)}:=h_F^{-1}(f_i(x_1)). \] 
		After passing to a subsequence we may assume that the intervals \(I_{f_i(x_1)}\) shrink toward one endpoint of the fiber over \(a\). Hence \(\Psi(f_i)|_{I_{x_1}}\) converges pointwise to a constant map; denote its value by \(a^\ast\). 
		
		On the other hand, if \(z\in \core(h_F)\) satisfies \(h_F(z)\neq x_1\), then by the semi-conjugacy relation the projection of \(\Psi(f_i)(z)\) to the quotient circle converges to \(b\). Passing to a further subsequence if necessary, we may assume that \(\Psi(f_i)(z)\) converges to some point \(b^\ast\) lying over \(h^{-1}_F(b)\). Since \(a\neq b\), the fibers over \(a\) and \(b\) are disjoint, and therefore \(a^\ast\neq b^\ast\). 
		
		Choose two distinct points \(u_1,u_2\in I_{x_1}\) and two distinct points \(z_1,z_2\in \core(h_F)\) such that \(h_F(z_j)\neq x_1\) for \(j=1,2\). Then \(\Psi(f_i)(u_1)\) and \(\Psi(f_i)(u_2)\) both converge to \(a^\ast\), while, after passing to a further subsequence if necessary, \(\Psi(f_i)(z_1)\) and \(\Psi(f_i)(z_2)\) converge to points \(b_1^\ast,b_2^\ast\) lying over \(h^{-1}_F(b)\). In particular, \(a^\ast\neq b_1^\ast\) and \(a^\ast\neq b_2^\ast\).
		
		It follows that no choice of a single exceptional point can make the sequence \(\{\Psi(f_i)\}\) converge pointwise to a constant on its complement. Hence \(\{\Psi(f_i)\}\) does not satisfy alternative~\ref{a.convergence_group} of the convergence-group theorem. Since it also does not converge to a homeomorphism, it does not satisfy alternative~\ref{b.convergence_group} either. 
		
		Passing to the quotient by \(\tau\), we obtain a sequence in the quotient action of \(H\) which still satisfies neither alternative~\ref{a.convergence_group} nor alternative~\ref{b.convergence_group} of the convergence-group theorem. This contradicts the fact that the quotient action of \(H\) is topologically conjugate to a subgroup of \(\psl(2,\RR)\). Therefore \(H\) is not topologically conjugate to a subgroup of \(\psl^{(k)}(2,\RR)\). Since \(k\) is the only possible lift order for \(F\), it follows that \(H\) is not topologically conjugate to a subgroup of \(\psl^{(j)}(2,\RR)\) for any \(j\geq 1\). 
	\end{proof}
	
	\begin{remark}
		Theorem~\ref{t.conical_examples_non_psl} is stated for trivial stabilizers because the
		conical-limit mechanism is incompatible with the parabolic regime arising from
		nontrivial stabilizers. Indeed, by Theorem~\ref{t.main_examples_2n}, in order for the
		amalgamated product to remain in the class of actions with at most \(2n\) fixed points,
		the stabilizer subgroup must consist of elements fixing exactly the \(n\) points
		\(x_1,\ldots,x_n\) and \(y_1,\ldots,y_n\). In the \(\psl^{(n)}(2,\RR)\)-like setting,
		elements with exactly \(n\) fixed points are precisely the parabolic ones. Hence the
		points \(x_i\) and \(y_i\) are parabolic fixed points of the original actions. But
		parabolic fixed points are not conical limit points, so the argument of Theorem~\ref{t.conical_examples_non_psl} cannot be applied in
		that setting; see \cite{BeardonMaskit74,Tukia98}.
	\end{remark}
	
	\bigskip
	
	{\small \subsection*{Acknowledgments}
		This article is based on work carried out during the author's PhD at Universit\'e Bourgogne
		Franche-Comt\'e, under the supervision of Christian Bonatti and Michele Triestino. The author is
		deeply grateful to both of them for their guidance, encouragement, and many valuable discussions.
		This work was partially supported by the project MATH AMSUD DGT -- Dynamical Group Theory
		(22-MATH-03) and the project ANR Gromeov (ANR-19-CE40-0007).
		.}

\end{document}